\documentclass[10pt]{amsart}
\usepackage{amssymb}
\usepackage{epsfig}
\usepackage{url}
\usepackage[T1]{fontenc}
\usepackage{color}
\usepackage{tikz}
\usetikzlibrary{arrows}
\theoremstyle{plain}

\newtheorem{thm}{Theorem}[section]
\newtheorem{cor}[thm]{Corollary}
\newtheorem{lem}[thm]{Lemma}
\newtheorem{prop}[thm]{Proposition}
\newtheorem{rem}[thm]{Remark}
\newtheorem{ques}[thm]{Question}
\newtheorem{conj}[thm]{Conjecture}
\newtheorem{exam}[thm]{Example}

\def\cal{\mathcal}
\def\bbb{\mathbb}
\def\op{\operatorname}
\renewcommand{\phi}{\varphi}

\newcommand{\N}{\bbb{N}}
\newcommand{\Z}{\bbb{Z}}
\newcommand{\Q}{\bbb{Q}}

\newcommand{\bs}{\backslash}

\begin{document}

\title[On some properties of the number of certain permutations]{On some properties of the number of permutations being products of pairwise disjoint $d$-cycles}
\author{Piotr Miska and Maciej Ulas}

\thanks{The research of the first author was partially supported by the grant of the Polish National Science Centre no. UMO-2018/29/N/ST1/00470}


\maketitle

\begin{abstract}
Let $d\geq 2$ be an integer. In this paper we study arithmetic properties of the sequence $(H_d(n))_{n\in\N}$, where $H_{d}(n)$ is the number of permutations in $S_{n}$ being products of pairwise disjoint cycles of a fixed length $d$. In particular we deal with periodicity modulo a given positive integer, behaviour of the $p$-adic valuations and various divisibility properties. Moreover, we introduce some related families of polynomials and study they properties. Among many results we obtain qualitative description of the $p$-adic valuation of the number $H_{d}(n)$ extending in this way earlier results of Ochiai and Ishihara, Ochiai, Takegehara and Yoshida.
\end{abstract}

\section{Introduction}\label{Section1}
We let $\N$ denote the set of non-negative integers, $\N_{+}$ the set of positive integers, $\mathbb{P}$ the set of prime numbers and finally we write $\N_{\geq k}$ for the set $\{n\in\N:\;n\geq k\}$. In the sequel we will also need the notion of the $p$-adic valuation of an integer, where $p\in\mathbb{P}$ is fixed. More precisely, if $n\in\Z$ then the number
$$
\nu_{p}(n):=\op{max}\{k\in\N:\;p^{k}\mid n\}
$$
is called the $p$-adic valuation of $n$. We also adopt the standard convention that $\nu_{p}(0)=+\infty$. From the definition we easily deduce that for each $n_{1}, n_{2}\in\Z$ the following properties hold:
$$
\nu_{p}(n_{1}n_{2})=\nu_{p}(n_{1})+\nu_{p}(n_{2})\quad\mbox{and}\quad \nu_{p}(n_{1}+n_{2})\geq \op{min}\{\nu_{p}(n_{1}),\nu_{p}(n_{2})\}.
$$
If $\nu_{p}(n_{1})\neq \nu_{p}(n_{2})$, then the inequality can be replaced by equality. Moreover, one can easily extend the notion of $p$-adic valuation to rational numbers $x=a/b, b\neq 0$, in the following way: $\nu_{p}(x)=\nu_{p}(a)-\nu_{p}(b)$.

Let $d\in\N_{\geq 2}$, $n\in\N$ and consider the number $H_{d}(n)$ of those $\sigma\in S_{n}$ which are products of pairwise disjoint cycles of length $d$. As usual, the identity permutation is also counted as a product of $0$ cycles of length $d$. In particular, if $d=p\in\mathbb{P}$ is a prime number, then $H_p(n)$ counts the number of solutions in $S_{n}$ of the equation $\sigma^{p}=id$.

In the paper \cite{AmdMoll} the authors initiated the study of arithmetic properties of the sequence $(H_{d}(n))_{n\in\N}$ for $d=2$. In this case, the number $H_{2}(n)$ is the number of involutions in the group $S_{n}$. In particular, they obtained several combinatorial identities, presented description of the $2$-adic valuation of $H_{2}(n)$ and gave precise information about the rates of growth of $H_2(n)$. The mentioned paper can be seen as a complement of case $p=2$ of the papers \cite{GraNew, Och, IOTY, Kim^2} concerning $p$-adic valuations of numbers $H_p(n)$, where $p$ is a prime number. The study of the sequence $(H_{p}(n))_{n\in\N}$ with $p$-prime, is quite natural because the number $H_{p}(n)$ counts the elements of order $p$ in the group $S_{n}$. However, it seems that there are no results concerning the general sequence $(H_{d}(n))_{n\in\N}$, where $d\in\N_{\geq 2}$ is not necessarily a prime. In particular, it should be stressed that there is no results on $p$-adic valuations of numbers $H_d(n)$, where $p$ is a prime number greater than $d$ (even in the case when $d$ is a prime). Our first aim is to fill this gap and present broad spectrum of results concerning various arithmetic properties of the sequence $(H_d(n))_{n\in\N}$ and also some families of polynomials related to them.

Let us introduce the content of the paper.

Section \ref{Section2} is devoted to basic properties of the general sequence $(H_{d}(n))_{n\in\N}$. In particular, we recall the standard recurrence relation and closed formula for our sequence. Moreover, we present the expression for the exponential generating function
$$
\cal{H}_d(x)=\sum_{n=0}^{\infty}\frac{H_{d}(n)}{n!}x^n.
$$
With the help of the function $\cal{H}_{d}(x)$ we obtain several interesting identities involving elements of our sequence and related sequences (such as roots of unity or Bernoulli numbers).

In Section \ref{Section3} we study periodicity properties of the sequences of remainders of numbers $H_d(n)$ modulo a given positive integer $c$. We start with the results for $c\in\{d+1,d+2\}$. We continue the work with the case of $c$ being of the form $d+k$, where $k$ is a fixed positive integer. Next we show that if $c$ is a power of a prime number not equal to $d$, then $c$ is a period of the sequence $(H_d(n)\pmod{c})_{n\in\N}$. At the end of the section, the result for composite $d$ and arbitrary $c$ or for $c$ co-prime to $d$ is given.

Section \ref{Section4} begins with a revision of a lower estimate of the $p$-adic valuation of the number $H_p(n)$:
$$\nu_p(H_p(kp^2+jp+i))\geq k(p-1)+j,\quad k\in\N, i,j\in\{0,...,p-1\},$$
where the prime number $p$ is fixed. We give a new proof of this known fact in order to obtain some results on periodicity of sequences $\left(\frac{H_p(n)}{p^{\beta_n}}\pmod{p^r}\right)_{n\in\N}$, where $r\in\N_+$ and $\beta_n$ is mentioned lower bound of $\nu_p(H_p(n))$ or $\beta_n=\nu_p(H_p(n))$. Moreover, we show that for each prime number $p$ and $j\in\{0,...,p-1\}$ there exists a $b_j\in\{0,...,p-1\}$ such that $H_p(kp^2+jp+b_j)=k(p-1)+j$ for any $k\in\N$. These results can be seen as a complement of the series of papers \cite{GraNew, Och, IOTY, Kim^2}.

In Section \ref{Section5} we describe the $p$-adic valuations of the numbers $H_d(n)$, where $p>d$. In order to do this, we show that the sequence $(H_d(n))_{n\in\N}$ is a restriction of a differentiable function $f_d:\Z_p\rightarrow\Z_p$, where $\Z_p$ is a ring of $p$-adic integers. Then, we prove that there exists a function $g_d:\Z_p\times p\Z_p\rightarrow\Q_p$ such that for any $x\in\Z_p$ and $h\in p\Z_p$ we have
$$f_d(x+h)=f_d(x)+hf'_d(x)+h^2g_d(x,h)$$
and $g_d(x,h)\in\Z_p$ if $d\geq 3$ or $p\mid f_2(x)$. At last, we obtain a qualitative description of the sequence $(\nu_p(H_d(n)))_{n\in\N}$ using Hensel's lemma.

Section \ref{Section6} is devoted to study of a family of polynomials $W_{d,m}(x)$ which are closely related to $m$-th derivative of the generating function $\cal{H}_{d}(x)$. More precisely, we have $\cal{H}^{(m)}_d(x)=W_{d,m}(x)\cal{H}_d(x)$. We show that $W_{d,m}(x)$ is a monic polynomial with integral coefficients and we give a formula for its coefficients in case of odd $d$. These coefficients are expressed in terms of (shifted) elements of the sequence $(H_{d}(n))_{n\in\N}$. Next, we compute the exponential generating function $\cal{W}_d(x,t)$ of the sequence $(W_{d,m}(x))_{m\in\N}$ and prove identities and congruences involving these polynomials. We finish the section with the proof of periodicity of the sequence $(x^{\left\lfloor\frac{m}{p}\right\rfloor p(1-p)}W_{p,m}(x)\pmod{p})_{m\in\N}$.

In Section \ref{Section7} we deal with divisors of the number $H_d(n)$. The section is divided into three parts. In the first part we show that for each $d\geq 2$ the set of prime divisors of numbers $H_d(n)$ is infinite. In the second part we compute the greatest common divisor of numbers $H_d(n)-1$ and $n$. As a consequence, we get that the set prime divisors of numbers $H_{d}(n)-1$, $n\in\N$, is whole $\bbb{P}$. The last part is devoted to the study of  values $2\leq a<b$, $n\in\N$ and $c\in\N_{\geq 2}$ such that $c$ divides both $H_a(n)$ and $H_b(n)$.

In Section \ref{Section8} we introduce the sequence of polynomials $(H_d(n,x))_{n\in\N}$ being a generalizations of the sequence $(H_d(n))_{n\in\N}$. Namely, $H_d(n,1)=H_d(n)$ and the coefficient near the $x^k$ in the polynomial $H_d(n,x)$ is the number of those permutations in $S_n$ which are products of pairwise disjoint $d$-cycles and have exactly $k$ fixed points. After that, we compute the remainders of polynomials $H_d(n,x)$ modulo $d$ and study coefficients of the polynomials $H_2(n-1,x)H_2(n+1,x)-H_2(n,x)^2$ and $H_d(n-d,x)H_d(n+d,x)-H_d(n,x)^2$. In particular, we prove that if $U(n,x)=H_2(n-1,x)H_2(n+1,x)-H_2(n,x)^2$, then $U(n,x)=V(n,x^2)$, where $V\in\Z[x]$ has degree $n-1$ and all the coefficients of $V$ near $x^{i}, i\neq 0, n-1$ are positive. Moreover, among some identities and congruences the most interesting result of this section is the result which says that the congruence $H_d(n-d,x)H_d(n+d,x)\equiv H_d(n,x)^2\pmod{d!}$ is true if and only if $d$ is a prime number or a square of a prime number.

Section \ref{Section9} collects some observations, questions and conjectures related to objects introducted in previous sections.

\section{First results}\label{Section2}
For $d\in\N_{\geq 2}$ we have $H_{d}(i)=1$ for $i\in\{0,1,\ldots,d-1\}$ by definition and for $n\geq d$ the standard combinatorial argument shows that
\begin{equation}\label{basicrec}
H_{d}(n)=H_{d}(n-1)+(n-1)_{(d-1)}H_{d}(n-d),
\end{equation}
where as usual, for $m\in\N$ the symbol $(u)_{(m)}$ denotes the falling factorial defined as
$$
(u)_{(m)}=u(u-1)\cdot\ldots\cdot (u-(m-1)).
$$
Indeed, let us fix $\sigma\in S_{d,n}$. If $\sigma(n)=n$ then $\sigma=\sigma_1$ for some $\sigma_1\in S_{d,n-1}$ and it can be chosen by $H_d(n-1)$ ways. If $\sigma(n)\neq n$ then $\sigma=\sigma_1\circ\pi$, where $\pi$ is a $d$-cycle containing $n$ and $\sigma_1$ is a product of pairwise disjoint $d$-cycles on a set $\{1,...,n\}\bs\mbox{supp}(\pi)$. Then $\sigma_1$ can be treated as a member of the set $S_{d,n-d}$. Thus, it can be chosen by $H_d(n-d)$ ways. Moreover, $\pi$ can be chosen in $(n-1)_{(d-1)}$ ways.

Furthermore, in the sequel we will also use the convention $H_{d}(n)=0$ for $n<0$. As a simple consequence of the recurrence formula we see that if $q\in\N$, with $q\mid (d-1)!$, then $H_{d}(n)\equiv 1\pmod{q}$ for each $n\in\N$. In particular, if $d$ is a composite number $>4$, then $H_{d}(n)\equiv 1\pmod{d}$. Moreover, if $p$ is a prime number $<d$, then $\nu_p(H_d(n))=0$ for each $n\in\N$.

It is well known that the exponential generating function for the sequence $(H_{d}(n))_{n\in\N}$ is given by
$$
\cal{H}_{d}(x)=\sum_{n=0}^{\infty}\frac{H_{d}(n)}{n!}x^{n}=e^{x+\frac{x^{d}}{d}}.
$$
In consequence, the exact expression for $H_{d}(n)$ is given by
\begin{equation}\label{exactform}
H_{d}(n)=\sum_{k=0}^{\lfloor\frac{n}{d}\rfloor}\frac{n!}{(n-dk)!k!}\frac{1}{d^{k}}.
\end{equation}
The knowledge of the closed form of exponential generating function allows us to obtain some identities involving numbers $H_d(n)$.

\begin{thm}
 The following identities are true.
 \begin{enumerate}
  \item
  $$
  \sum_{k=0}^{n}{n\choose k}(-1)^{n-k}H_{d}(k)=\begin{cases}
\begin{array}{lll}
0, &  & \mbox{if}\;n\not\equiv 0\pmod{d}\\
\frac{(md)!}{m!d^{m}}, &  & \mbox{if}\;n=md.
\end{array}
\end{cases}
  $$
  \item For $n\in\N_{+}$ and $d\in\N_{\geq 3}$ an odd number we have
  $$
  \sum_{k=0}^{n}{n\choose k}(-1)^{k}H_{d}(k)H_{d}(n-k)=0.
  $$
  \item For $n\in\N_{+}$ and $d\in\N_{\geq 2}$ an even number we have
   $$
  \sum_{k=0}^{n}{n\choose k}(-1)^{k}H_{d}(k)H_{d}(n-k)=\begin{cases}
\begin{array}{lll}
0, &  & \mbox{if}\;n\not\equiv 0\pmod{d}\\
\frac{(md)!}{m!\left(\frac{d}{2}\right)^{m}},&  & \mbox{if}\;n=md.
\end{array}
\end{cases}
  $$
  \item Let $\zeta_{d}$ be a $d$-th root of unity. Then
  $$
  H_{d}(n)=\zeta_{d}^{-n}\sum_{k=0}^{n}{n\choose k}(\zeta_{d}-1)^{k}H_{d}(n-k).
  $$
  \item Let $\zeta_{2d}$ be a $2d$-th root of unity satisfying $\zeta_{2d}^d=-1$. Then
  $$
  \sum_{k=0}^{n}{n\choose k}\zeta_{2d}^{k}H_{d}(k)H_{d}(n-k)=(1+\zeta_{2d})^{n}.
  $$
  \item Let $d$ be an even number and $(B_{n})_{n\in\N}$ denote the sequence of Bernoulli numbers. Then
  $$
  \sum_{k=0}^{\lfloor\frac{n}{2}\rfloor}\binom{n}{2k}\frac{2^{n-2k}}{2k+1}B_{n-2k}H_{d}(2k+1)=(-1)^{n}H_{d}(n).
  $$
  \item For $n\in\N$ we have the identity
  $$
  H_{d}(n)=1+\sum_{i=1}^{n}(i-1)_{(d-1)}H_{d}(i-d).
  $$
 \end{enumerate}
\end{thm}
\begin{proof}
The identities (1)-(6) can be easily proved with the help of generating functions technique. Indeed, it is well known that for
$$
f(x)=\sum_{n=0}^{\infty}\frac{a_{n}}{n!}x^{n},\quad g(x)=\sum_{n=0}^{\infty}\frac{b_{n}}{n!}x^{n}
$$
we have $f(x)g(x)=\sum_{n=0}^{\infty}\frac{c_{n}}{n!}x^{n}$ with
$$
c_{n}=\sum_{k=0}^{n}\binom{n}{k}a_{k}b_{n-k}.
$$

In order to get the first identity we write $f(x)=\cal{H}_{d}(x)$ and $g(x)=e^{-x}$ and get $f(x)g(x)=e^{\frac{x^{d}}{d}}$. Comparing the coefficients on the both sides of this identity we get the result.

In order to get the second identity we put $f(x)=\cal{H}_{d}(x), g(x)=\cal{H}_{d}(-x)$. Because $d\geq 3$ is odd then $f(x)g(x)=1$ and hence the result.

The third identity can be proved in the same way. Taking $d$ even and $f(x)=\cal{H}_{d}(x), g(x)=\cal{H}_{d}(-x)$ we get $f(x)g(x)=e^{2\frac{x^{d}}{d}}$ and hence the result.

In order to get the fourth identity we put $f(x)=e^{(1-\zeta_{d})x}, g(x)=\cal{H}_{d}(\zeta_{d}x)$. Then $f(x)g(x)=\cal{H}_{d}(x)$ and the identity follows.

The fifth identity follows by taking $f(x)=\cal{H}_{d}(\zeta_{2d}x), g(x)=\cal{H}_{d}(x)$. In this case $f(x)g(x)=e^{(1+\zeta_{2d})x}$ and we get required identity.

In order to get the sixth identity we put
$$
f(x)=2x/(e^{2x}-1)=\sum_{n=0}^{\infty}\frac{2^{n}B_{n}}{n!}x^{n},
$$
i.e., $f$ is the exponential generating function of the sequence $(2^{n}B_{n})_{n\in\N}$, and
$$
g(x)=\frac{1}{2x}(\cal{H}_{d}(x)-\cal{H}_{d}(-x))=\sum_{n=0}^{\infty}\frac{A_{n}}{n!}x^{n},
$$
where $A_{n}=0$ for $n$ odd and $A_{n}=H_{d}(n+1)/(n+1)$ for $n$ even. With $f, g$ chosen in this way we get $f(x)g(x)=\cal{H}_{d}(-x)$ and comparing the coefficients near $x^{n}$ on both sides of this identity we get the result.

The last identity is a consequence of the recurrence relation (\ref{basicrec}). Indeed, rewriting (\ref{basicrec}) as $H_{d}(m)-H_{d}(m-1)=(m-1)_{(d-1)}H_{d}(m-d)$ and summing from $m=1$ to $n$ on both sides we get telescoping sum on the left and get
$$
H_{d}(n)-H_{d}(0)=\sum_{i=1}^{n}(i-1)_{(d-1)}H_{d}(i-d).
$$
The result follows.
\end{proof}

\section{Periodicity of the numbers $H_d(n)$ modulo a given positive integer}\label{Section3}

We know that if $c<d$ or $c=d>4$ is a composite number then the sequence $(H_d(n)\pmod{c})_{n\in\N}$ is constant and equal to $1$. In this section we present some results concerning periodicity of the sequences $(H_d(n)\pmod{c})_{n\in\N}$ when $c>d$. We start with cases $c\in\{d+1,d+2\}$. As we will see, the behaviour of the sequence $(H_d(n)\pmod{c})_{n\in\N}$ depends mainly on the question whether $c$ is a prime or not.

\begin{prop}
Let $p$ be an odd prime number. Then for each $n\in\N$ we have
\begin{align}\label{Hp-1modp}
H_{p-1}(n)\equiv
\begin{cases}
1,\mbox{ if } n\not\equiv -1\pmod{p}\\
2,\mbox{ if } n\equiv -1\pmod{p}
\end{cases}
\pmod{p}.
\end{align}
In particular, $\nu_p(H_{p-1}(n))=0$ for any $n\in\N$.

If additionally $p>3$, then
\begin{align}\label{Hp-2modp}
H_{p-2}(n)\equiv
\begin{cases}
\begin{array}{lll}
1,            &&\mbox{ if } n\not\equiv -1,-2\pmod{p}\\
\frac{p+1}{2},&&\mbox{ if } n\equiv -2\pmod{p}\\
\frac{p+3}{2},&&\mbox{ if } n\equiv -1\pmod{p}
\end{array}
\end{cases}
\pmod{p}.
\end{align}
As a result, $\nu_p(H_{p-2}(n))=0$ for any $n\in\N$.
\end{prop}

\begin{proof}
We start with the computation of $H_{p-1}(n)\pmod{p}$. We proceed by induction on $n\in\N$. For $n\in\{0,...,p-2\}$ we have $H_{p-1}(n)=1$. Let us fix now $N\geq p-1$ and assume that congruence (\ref{Hp-1modp}) is satisfied for each $n<N$. If $N\not\equiv -1,0\pmod{p}$ then by (\ref{basicrec}) we have
\begin{align*}
H_{p-1}(N) =H_{p-1}(N-1)+(N-1)_{(p-2)}H_{p-1}(N-p+1)\equiv H_{p-1}(N-1)\equiv 1\pmod{p},
\end{align*}
since in the product of numbers $N-1,...,N-p+2$ there is a number divisible by $p$ and $N-1\not\equiv -1\pmod{p}$, which means that $H_{p-1}(N-1)\equiv 1\pmod{p}$. If $N\equiv 0\pmod{p}$, then
\begin{align*}
H_{p-1}(N) &=H_{p-1}(N-1)+(N-1)_{(p-2)}H_{p-1}(N-p+1)\\
&\equiv H_{p-1}(N-1)+(p-1)!H_{p-1}(N-p+1)\equiv 2-1=1\pmod{p},
\end{align*}
since $N-1\equiv -1\pmod{p}$, $N-p+1\equiv 1\not\equiv -1\pmod{p}$ ($p>2$) and by Wilson's theorem $(p-1)!\equiv -1\pmod{p}$. If $N\equiv -1\pmod{p}$, then
\begin{align*}
H_{p-1}(N) &=H_{p-1}(N-1)+(N-1)_{(p-2)}H_{p-1}(N-p+1)\\
&\equiv H_{p-1}(N-1)+(p-2)!H_{p-1}(N-p+1)\equiv 1+1=2\pmod{p},
\end{align*}
since $N-1$ and $N-p+1$ are not congruent to $-1$ modulo $p$ and $(p-2)!\equiv 1\pmod{p}$ as an immediate consequence of Wilson's theorem.

The proof of congruence (\ref{Hp-2modp}) is analogous. We proceed by induction on $n\in\N$. For $n\in\{0,...,p-3\}$ we have $H_{p-2}(n)=1$. Let us fix now $N\geq p-2$ and assume that congruence (\ref{Hp-2modp}) is satisfied for each $n<N$. If $N\not\equiv -2,-1,0\pmod{p}$, then by (\ref{basicrec}) we have
\begin{align*}
H_{p-2}(N) =H_{p-2}(N-1)+(N-1)_{(p-3)}H_{p-2}(N-p+2)\equiv H_{p-2}(N-1)\equiv 1\pmod{p},
\end{align*}
since in the product of numbers $N-1,...,N-p+3$ there is a number divisible by $p$ and $N-1\not\equiv -2,-1\pmod{p}$, which means that $H_{p-1}(N-1)\equiv 1\pmod{p}$. If $N\equiv 0\pmod{p}$, then
\begin{align*}
H_{p-2}(N) &=H_{p-2}(N-1)+(N-1)_{(p-3)}H_{p-2}(N-p+2)\\
&\equiv H_{p-2}(N-1)+\frac{(p-1)!}{2}H_{p-2}(N-p+2)\equiv \frac{p+3}{2}-\frac{1}{2}=\frac{p}{2}+1\equiv 1\pmod{p},
\end{align*}
since $N-1\equiv -1\pmod{p}$, $N-p+2\equiv 2\not\equiv -2,-1\pmod{p}$ ($p\geq 5$) and by Wilson's theorem $(p-1)!\equiv -1\pmod{p}$. If $N\equiv -1\pmod{p}$, then
\begin{align*}
H_{p-2}(N) &=H_{p-2}(N-1)+(N-1)_{(p-3)}H_{p-2}(N-p+2)\\
&\equiv H_{p-2}(N-1)+(p-2)!H_{p-2}(N-p+2)\equiv \frac{p+1}{2}+1=\frac{p+3}{2}\pmod{p},
\end{align*}
since $N-1\equiv -2\pmod{p}$, $N-p+2\equiv 1\not\equiv -2,-1\pmod{p}$ ($p\geq 5$) and $(p-2)!\equiv 1\pmod{p}$ as an immediate consequence of Wilson's theorem. If $N\equiv -2\pmod{p}$, then
\begin{align*}
H_{p-2}(N) &=H_{p-2}(N-1)+(N-1)_{(p-3)}H_{p-2}(N-p+2)\\
&\equiv H_{p-2}(N-1)+(p-3)!H_{p-2}(N-p+2)\equiv 1-\frac{1}{2}=\frac{1}{2}\pmod{p},
\end{align*}
since $N-1\equiv -3\not\equiv -2,-1\pmod{p}$, $N-p+2\equiv 0\not\equiv -2,-1\pmod{p}$ ($p\geq 5$) and $(p-3)!\equiv -\frac{1}{2}\pmod{p}$ as an easy consequence of Wilson's theorem.
\end{proof}

\begin{prop}
Let $d$ be a positive integer at least $5$. If $d+1$ is a composite number, then for each $n\in\N$ we have $H_{d}(n)\equiv 1\pmod{d+1}$.

Let $d$ be a positive integer at least $4$. If $d+2$ is a composite number, then for each $n\in\N$ we have $H_{d}(n)\equiv 1\pmod{d+2}$.
\end{prop}

\begin{proof}
Since $H_d(n)=1$ for $n\in\{0,...,d-1\}$, $H_{d}(n)=H_{d}(n-1)+(n-1)_{(d-1)}H_{d}(n-d)$ for $n\geq d$ and $(d-1)!\mid (n-1)_{(d-1)}$ for any $n\in\N$, it suffices to prove that $d+1\mid (d-1)!$ in the first part of the statement of our theorem and $d+2\mid (d-1)!$ in the second part.

At first, we show $d+1\mid (d-1)!$ if $d\geq 5$ and $d+1$ is a composite number. Let $q$ be the least divisor of $d+1$ greater than $1$. Then $\frac{d+1}{q}\leq\frac{d+1}{2}<d-1$, since $d\geq 5$. If $q<\frac{d+1}{q}$, then $q$ and $\frac{d+1}{q}$ are two distinct positive integers appearing as factors in $(d-1)!$, thus $d+1\mid (d-1)!$. If $q=\frac{d+1}{q}$ then $d+1=q^2$. Since $d+1\geq 6$ thus $q\geq 3$. Then $2q<q^2-2=d-1$ and $q, 2q$ appear as factors in $(d-1)!$. Hence $d+1\mid (d-1)!$.

Now we prove $d+2\mid (d-1)!$ if $d\geq 4$ and $d+2$ is a composite number. Let $q$ be the least divisor of $d+2$ greater than $1$. Then $\frac{d+2}{q}\leq\frac{d+2}{2}\leq d-1$, since $d\geq 4$. If $q<\frac{d+2}{q}$, then $q$ and $\frac{d+2}{q}$ are two distinct positive integers appearing as factors in $(d-1)!$, thus $d+2\mid (d-1)!$. If $q=\frac{d+2}{q}$ then $d+2=q^2$. Since $d+2\geq 6$ thus $q\geq 3$. Then $2q\leq q^2-3=d-1$ and $q, 2q$ appear as factors in $(d-1)!$. Hence $d+1\mid (d-1)!$.
\end{proof}

According to the last proposition and numerical computations, we prove that for a fixed positive integer $k$ and  sufficiently large positive integer $d$ such that $d+k$ is a composite number, the sequence $(H_d(n)\pmod{d+k})_{n\in\N}$ is constant and equal to $1$.

\begin{thm}
Let us fix $k\in\N_+$. Then for each positive integer $d\geq\max\{k+2,3+2\sqrt{k+2}\}$ such that $d+k$ is a composite number there holds $H_d(n)\equiv 1\pmod{d+k}$ for all $n\in\N$.
\end{thm}

\begin{proof}
Similarly as in the proof of the previous proposition, it suffices to show that $d+k\mid (d-1)!$. Let $q$ be the least divisor of $d+k$ greater than $1$. Then $\frac{d+k}{q}\leq\frac{d+k}{2}\leq d-1$, since $d\geq k+2$. If $q<\frac{d+k}{q}$, then $q$ and $\frac{d+k}{q}$ are two distinct positive integers appearing as factors in $(d-1)!$, thus $d+k\mid (d-1)!$. If $q=\frac{d+k}{q}$ then $d+k=q^2$. Since $3+2\sqrt{k+2}=(1+\sqrt{k+2})^2-k$ and $d\geq (1+\sqrt{k+2})^2-k$, we have $q=\sqrt{d+k}\geq 1+\sqrt{k+2}$. As a consequence, $2q\leq q^2-k-1=d-1$ and $q, 2q$ appear as factors in $(d-1)!$. Hence $d+k\mid (d-1)!$.
\end{proof}

For our next result, we need the following auxiliary lemma.

\begin{lem}\label{ineqval}
Let $p$ be a prime number and $d$ be a composite positive integer divisible by $p$. Assume that $(p,d)\neq (2,4)$. Then $\frac{d-2}{p}\geq \nu_p(d)$.
\end{lem}

\begin{proof}
If $\nu_p(d)=1$ then $d\geq 2p$ as $d$ is a composite number. We thus have
$$\frac{d-2}{p}\geq\frac{2p-2}{p}=2-\frac{2}{p}\geq 1$$
as $p\geq 2$. If $p\geq 3$ and $\nu_p(d)\geq 2$ or $p=2$ and $\nu_2(d)\geq 3$ then
$$\frac{d}{p}\geq\frac{p^{\nu_p(d)}}{p}=p^{\nu_p(d)-1}\geq \nu_p(d)+1\geq \nu_p(d)+\frac{2}{p},$$
which implies $\frac{d-2}{p}\geq \nu_p(d)$. If $p=2$, $\nu_2(d)=2$ and $d\neq 4$ then $d\geq 3\cdot 4$ and
$$\frac{d-2}{2}\geq\frac{12-2}{2}=5>2=\nu_2(d).$$
Our lemma is proved.
\end{proof}

We are ready to prove the main result of the section. It concerns periodicity of the sequence $(H_d(n)\pmod{p^r})_{n\in\N}$, where $p$ is a prime number not equal to $d$ and $r$ is any positive integer. More precisely, we have the following

\begin{thm}
Let $d,r$ be positive integers with $d\geq 2$. Let $p$ be a prime number not equal to $d$. Assume that $(d,p,r)\neq (4,2,2)$. Then the sequence $(H_d(n)\pmod{p^r})_{n\in\N}$ is periodic of period $p^r$. Moreover, if $p>d$, then $p^r$ is the basic period of the sequence $(H_d(n)\pmod{p^r})_{n\in\N}$.
\end{thm}

\begin{proof}
Let us notice that for the proof that $p^r$ is a period of $(H_d(n)\pmod{p^r})_{n\in\N}$ we only need to show that
\begin{equation}\label{modpr}
H_d(p^r)\equiv H_d(0)=1\pmod{p^r}.
\end{equation}
Then, using (\ref{basicrec}) and a simple induction on $n\in\N$, we prove that $H_d(n+p^r)\equiv H_d(n)\pmod{p^r}$.

In order to prove (\ref{modpr}) we will use the exact formula $$H_d(p^r)=\sum_{k=0}^{\lfloor\frac{p^r}{d}\rfloor}\frac{p^r!}{(p^r-dk)!k!}\frac{1}{d^{k}}$$ and show that for $k>1$, the $k$-th summand in the above sum has the $p$-adic valuation greater than $r$. We write this summand:
$$\frac{p^r!}{(p^r-dk)!k!}\frac{1}{d^{k}}=p^r\frac{(p^r-1)!}{(p^r-dk)!k!}\frac{1}{d^{k}}=p^r\frac{(p^r-1)_{(dk-1)}}{k!}\frac{1}{d^{k}}.$$
It suffices to show, that $\nu_p\left(\frac{(p^r-1)_{(dk-1)}}{k!}\frac{1}{d^{k}}\right)\geq 0$. If $\nu_p(d)=0$ then the mentioned inequality is true. We thus assume that $\nu_p(d)>0$. Since $(dk-1)!\mid (p^r-1)_{(dk-1)}$, we have
$$\nu_p\left(\frac{(p^r-1)_{(dk-1)}}{k!}\frac{1}{d^{k}}\right)\geq \nu_p\left(\frac{(dk-1)!}{k!}\frac{1}{d^{k}}\right)=\nu_p((dk-1)_{((d-1)k-1)})-k\nu_p(d).$$
We estimate $\nu_p((dk-1)_{((d-1)k-1)})$ from below by number of multiples of $p$ among factors $dk-1, dk-2, ..., k+1$ and use Lemma \ref{ineqval} to obtain the following.
$$\nu_p((dk-1)_{((d-1)k-1)})>\frac{(d-1)k-1}{p}-1\geq\frac{k(d-2)}{p}-1\geq k\nu_p(d)-1.$$
This implies $\nu_p\left(\frac{(p^r-1)_{(dk-1)}}{k!}\frac{1}{d^{k}}\right)\geq 0$ and finishes the proof of (\ref{modpr}).

We need to show that $p^r$ is the basic period of the sequence $(H_d(n)\pmod{p^r})_{n\in\N}$ provided that $p>d$. Let us assume by contrary that $p^{r-1}$ is a period of $(H_d(n)\pmod{p^r})_{n\in\N}$. If $r=1$ then we directly obtain a contradiction, since the sequence $(H_d(n)\pmod{p})_{n\in\N}$ is not constant. Indeed, $H_d(d)=1+(d-1)!\not\equiv 1\pmod{p}$ as $d<p$. If $r>1$, then we consider the congruences
\begin{align*}
H_d(p^{r-1}+d)\equiv &H_d(d)\pmod{p^r},\\
H_d(p^{r-1}+d+1)\equiv &H_d(d+1)\pmod{p^r},
\end{align*}
that are equivalent to
\begin{align*}
H_d(p^{r-1}+d-1)+(p^{r-1}+d-1)_{(d-1)}H_d(p^{r-1})\equiv & H_d(d-1)+(d-1)!H_d(0)\pmod{p^r},\\
H_d(p^{r-1}+d)+(p^{r-1}+d)_{(d-1)}H_d(p^{r-1}+1)\equiv & H_d(d)+d!H_d(1)\pmod{p^r}.
\end{align*}
Since $H_d(p^{r-1}+d-1)\equiv H_d(d-1)\pmod{p^r}$, $H_d(p^{r-1}+d)\equiv H_d(d)\pmod{p^r}$ and $H_d(p^{r-1})\equiv H_d(0)\equiv H_d(p^{r-1}+1)\equiv H_d(1)\equiv 1\pmod{p^r}$, we get the following:
\begin{align*}
(p^{r-1}+d-1)_{(d-1)}\equiv & (d-1)!\pmod{p^r},\\
(p^{r-1}+d)_{(d-1)}\equiv & d!\pmod{p^r},
\end{align*}
and equivalently
\begin{align*}
(d-1)!+p^{r-1}(d-1)!\sum_{j=1}^{d-1}\frac{1}{j}\equiv & (d-1)!\pmod{p^r},\\
d!+p^{r-1}d!\sum_{j=2}^d\frac{1}{j}\equiv & d!\pmod{p^r}.
\end{align*}
After simplification (where we divide by $(d-1)!$ and $d!$, respectively, and these numbers are coprime to $p$) we obtain two congruences
\begin{align*}
\sum_{j=1}^{d-1}\frac{1}{j}\equiv & 0\pmod{p},\\
\sum_{j=2}^d\frac{1}{j}\equiv & 0\pmod{p}.
\end{align*}
Subtracting the second congruence from the first one, we get $1-\frac{1}{d}\equiv 0\pmod{p}$, or equivalently $\frac{d-1}{d}\equiv 0\pmod{p}$. This is impossible since $0<d-1<p$. Hence $p^r$ must be the basic period of the sequence $(H_d(n)\pmod{p^r})_{n\in\N}$.
\end{proof}

\begin{rem}
In the case of $(d,p,r)=(4,2,2)$ the sequence $(H_4(n)\pmod{4})_{n\in\N}$ has the basic period $8$. More precisely, we have
\begin{align*}
H_{4}(n)\equiv
\begin{cases}
1,\mbox{ if } n\equiv 0,1,2,3\pmod{8}\\
3,\mbox{ if } n\equiv 4,5,6,7\pmod{8}
\end{cases}
\pmod{4}.
\end{align*}
\end{rem}

\begin{proof}
We prove by induction on $m\in\N$ that
\begin{align}\label{H4mod8}
H_4(8m+i)\equiv
\begin{cases}
1,\mbox{ if } i\in\{0,1,2,3\}\\
3,\mbox{ if } i\in\{4,5,6,7\}
\end{cases}
\pmod{4}.
\end{align}
We check that congruence (\ref{H4mod8}) holds for $m=0$ and $i\in\{0,1,2,3,4,5,6,7\}$. Assume now that it is satisfied by some $m\in\N$. By direct calculations we show (\ref{H4mod8}) for $m+1$.
\end{proof}

\begin{exam}
If $p<d$, then it is possible that $p^r$ is not the basic period of the sequence $(H_d(n)\pmod{p^r})_{n\in\N}$. For example, if $d=5$, $p=2$, $r=4$, then we obtain a sequence $(H_5(n)\pmod{16})_{n\in\N}$, which has a period $8<2^4$.

Another example is for $d=10$, $p=3$, $r=5$. The basic period of the sequence $(H_{10}(n)\pmod{243})_{n\in\N}$ is $27<3^5$.
\end{exam}

Based on the above considerations we can give precise value of the period of the sequence $(H_d(n)\pmod{c})_{n\in\N}$ provided that $d$ is a composite number or $d\nmid c$.

\begin{cor}\label{Hdnmodc}
Let $c,d\in\N_+$ and $d>1$. Assume additionally that $d$ is a composite number or $d\nmid c$. Then the sequence $(H_d(n)\pmod{c})_{n\in\N}$ is periodic of period $2c$ if $d=4$ and $\nu_2(c)=2$ and $c$ otherwise.
\end{cor}

\section{The behaviour of the $p$-adic valuation of $H_{p}(n)$ for some $p$}\label{Section4}

Amdeberhan and Moll were able to compute the exact value of the 2-adic valuation of $H_{2}(n)$ for $n\in\N$. Quite unexpected they proved that for given $i\in\{0,1,2,3\}$ the sequence $(\nu_{2}(H_{2}(4n+i))-n)_{n\in\N}$ is constant. This raises natural question whether the sequence of $p$-adic valuations of $H_{p}(n)$ can share similar property. In fact, this question was studied earlier by Ochiai in \cite{Och}. In the cited paper the reader can find the formulae for the $p$-adic valuation of the numbers $H_p(n)$, $n\in\N$, where $p$ is a prime number at most equal to $11$. Actually, Ochiai in his paper gave a qualitative description of the sequence $(\nu_p(H_p(n)))_{n\in\N}$ for any prime number $p\leq 23$. Further, Ishihara, Ochiai, Takegahara and Yoshida extend this result for any prime number $p$ in \cite{IOTY}. In particular, the following inequality holds true
\begin{align}\label{lowerestimvpHpn}
\nu_p(H_p(n))\geq \left\lfloor\frac{n}{p^2}\right\rfloor\cdot (p-1)+\left(\left\lfloor\frac{n}{p}\right\rfloor\pmod{p}\right)
\end{align}
for all $n\in\N$, where the equality holds if $p\mid\left\lfloor\frac{n}{p}\right\rfloor$. This inequality was proved for $p=2$ by Chowla, Herstein and Moore in \cite{CHM}. In general, this inequality was proved by Grady and Newman in \cite{GraNew}. Moreover, in 2010 Kim and Kim gave a combinatorial proof of the above inequality. Despite this facts we would like to give the proof of (\ref{lowerestimvpHpn}) different than the ones presented by mentioned authors. In consequence, we will be able to prove the following two results which have not appeared yet in the literature.

\begin{thm}\label{Hpperiod}
For each odd prime number $p$ the sequence $$\left((-1)^{\left\lfloor\frac{n}{p^2}\right\rfloor}\frac{H_p(n)}{p^{\left\lfloor\frac{n}{p^2}\right\rfloor\cdot (p-1)+\left(\left\lfloor\frac{n}{p}\right\rfloor\pmod{p}\right)}}\pmod{p}\right)_{n\in\N}$$ is periodic with the basic period $p^2$. Moreover, the sequence $$\left(\frac{H_p(n)}{p^{\left\lfloor\frac{n}{p^2}\right\rfloor\cdot (p-1)+\left(\left\lfloor\frac{n}{p}\right\rfloor\pmod{p}\right)}}\pmod{p}\right)_{n\in\N}$$ has basic period equal to $2p^2$.
\end{thm}

\begin{prop}\label{vpHpeq}
For each prime number $p$ there are numbers $b_0,...,b_{p-1}\in\{0,...,p-1\}$ such that $b_0=p-1$, $b_{j-1}-b_j\in\{0,1\}$ for $j\in\{1,...,p-1\}$ and $\nu_p(H_p(kp^2+jp+b_j))=k(p-1)+j$ for all $k\in\N$ and $j\in\{0,...,p-1\}$. In particular, $b_j\geq p-j-1$.
\end{prop}

We start with one lemma. We omit its proof as it can be found in \cite[Lemma 2.1]{IOTY}.

\begin{lem}\label{sumprodrecip}
If $p$ is an odd prime number and $k$ is a positive integer then
\begin{align*}
\sum_{1\leq l_k\leq ...\leq l_1\leq p-1} \frac{1}{l_1\cdot ...\cdot l_k}\equiv
\begin{cases}
1,&\mbox{ when } p-1\mid k\\
0,&\mbox{ when } p-1\nmid k
\end{cases}\pmod{p}.
\end{align*}
\end{lem}

At this moment, we give a result on congruences involving numbers $H_p(n)$, $n\in\N$. This result is crucial in the proof of Theorem \ref{Hpperiod}.

\begin{thm}\label{Hpmodp}
For odd prime number $p$, a non-negative integer $k$ and integers $i,j\in\{0,1,...,p-1\}$ we have
\begin{align*}
H_p(kp^2)\equiv & (-1)^kp^{k(p-1)}\pmod{p^{k(p-1)+1}},\\
H_p(kp^2+jp)\equiv & 0\pmod{p^{k(p-1)+j}},\\
H_p(kp^2+jp+i)\equiv & H_p(kp^2+jp)\\
& +\sum_{t=1}^j (-1)^t(j)_{(t)}p^tH_p(kp^2+(j-t)p)\sum_{1\leq l_t\leq ...\leq l_1\leq i}\frac{1}{l_1\cdot ...\cdot l_t}\pmod{p^{k(p-1)+j+1}}.
\end{align*}
\end{thm}

\begin{proof}
We prove the congruences concerning numbers $H_p(n)$ for $n\in\N$. We proceed by induction on $n$. We write $n=kp^2+jp+i$ with $k\in\N$ and $i,j\in\{0,...,p-1\}$. If $k=j=0$, then the statement of the theorem holds for all $i\in\{0,...,p-1\}$.

Let us assume that $n=kp^2$ for some $k\in\N_+$. By induction hypothesis for $(k-1)p^2+(p-1)p+p-1$ and $(k-1)p^2+(p-1)p$, Wilson's theorem and Lemma \ref{sumprodrecip} we obtain
\begin{align*}
& H_p(kp^2) = H_p((k-1)p^2+(p-1)p+p-1)+(kp^2-1)_{(p-1)}H_p((k-1)p^2+(p-1)p)\\
& \equiv H_p((k-1)p^2+(p-1)p+p-1)+(p-1)!H_p((k-1)p^2+(p-1)p)\\
& \equiv H_p((k-1)p^2+(p-1)p)+\sum_{t=1}^{p-1} (-1)^t(p-1)_{(t)}p^tH_p((k-1)p^2+(p-1-t)p)\sum_{1\leq l_t\leq ...\leq l_1\leq p-1}\frac{1}{l_1\cdot ...\cdot l_t}\\
& - H_p((k-1)p^2+(p-1)p)\\
& \equiv (-1)^{p-1}\cdot(-1)p^{p-1}H_p((k-1)p^2)\equiv -p^{p-1}(-1)^{k-1}p^{(k-1)(p-1)}=(-1)^kp^{k(p-1)}\pmod{p^{k(p-1)+1}}.
\end{align*}

Assume now that $n=kp^2+jp$ for $k\in\N$ and $j\in\{1,...,p-1\}$. By the recurrence (\ref{basicrec}), induction hypothesis for $n=kp^2+(j-1)p+p-1$, Lemma \ref{sumprodrecip} and Wilson's theorem we have
\begin{align*}
& H_p(kp^2+jp)=H_p(kp^2+(j-1)p+p-1)+(kp^2+(j-1)p+p-1)_{(p-1)}H_p(kp^2+(j-1)p)\\
& \equiv H_p(kp^2+(j-1)p)+\sum_{t=1}^{j-1} (-1)^t(j-1)_{(t)}p^tH_p(kp^2+(j-1-t)p)\sum_{1\leq l_t\leq ...\leq l_1\leq p-1}\frac{1}{l_1\cdot ...\cdot l_t}\\
& + (p-1)!H_p(kp^2+(j-1)p)\\
& \equiv H_p(kp^2+(j-1)p)-H_p(kp^2+(j-1)p)=0\pmod{p^{k(p-1)+j}},
\end{align*}
since $p^{k(p-1)+j-1-t}\mid H_p(kp^2+(j-1-t)p)$.

At this moment we assume that $n=kp^2+jp+i$ for $k\in\N$ and $i,j\in\{0,...,p-1\}$. If $i=0$, then the third congruence obviously is true. For $i>0$, by (\ref{basicrec}), induction hypothesis for $n=kp^2+jp+i-1$ and Wilson's theorem we have
\begin{align*}
& H_p(kp^2+jp+i)=H_p(kp^2+jp+i-1)+(kp^2+jp+i-1)_{(p-1)}H_p(kp^2+(j-1)p+i)\\
& \equiv H_p(kp^2+jp+i-1)+\frac{jp\cdot(p-1)!}{i}H_p(kp^2+(j-1)p+i)\\
& \equiv H_p(kp^2+jp+i-1)-\frac{jp}{i}H_p(kp^2+(j-1)p+i)\\
& \equiv H_p(kp^2+jp)+\sum_{t=1}^j (-1)^t(j)_{(t)}p^tH_p(kp^2+(j-t)p)\sum_{1\leq l_t\leq ...\leq l_1\leq i-1}\frac{1}{l_1\cdot ...\cdot l_t}\\
& -\frac{pj}{i}H_p(kp^2+(j-1)p)+\sum_{t=1}^{j-1} (-1)^{t+1}(j)_{(t+1)}p^{t+1}H_p(kp^2+(j-1-t)p)\sum_{1\leq l_t\leq ...\leq l_1\leq i}\frac{1}{l_1\cdot ...\cdot l_t\cdot i}\\
& \equiv H_p(kp^2+jp)+\sum_{t=1}^j (-1)^t(j)_{(t)}p^tH_p(kp^2+(j-t)p)\sum_{1\leq l_t\leq ...\leq l_1\leq i}\frac{1}{l_1\cdot ...\cdot l_t}\pmod{p^{k(p-1)+j+1}}.
\end{align*}

\end{proof}

We are ready to prove the lower bound for the $p$-adic valuation of $H_p(n)$.

\begin{cor}\label{vpHpnfrombelow}
For odd prime number $p$, non-negative integer $k$ and integers $i,j\in\{0,1,...,p-1\}$ we have $\nu_p(H_p(kp^2+jp+i))\geq k(p-1)+j$, where the inequality is in fact equality for $j=0$. In other words, $$\nu_p(H_p(n))\geq \left\lfloor\frac{n}{p^2}\right\rfloor\cdot (p-1)+\left(\left\lfloor\frac{n}{p}\right\rfloor\pmod{p}\right)$$ for all $n\in\N$.
\end{cor}

\begin{proof}
We only need to prove that $\nu_p(H_p(kp^2+jp+i))=k(p-1)+j$ for $j=0$. We proceed by induction on $i\in\{0,...,p-1\}$. The equality $\nu_p(H_p(kp^2+i))=k(p-1)$ holds for $i=0$ by Theorem \ref{Hpmodp}. If $i>0$, then we write
$$H_p(kp^2+i)=H_p(kp^2+i-1)+(kp^2+i-1)_{(p-1)}H_p((k-1)p^2+(p-1)p+i).$$
Then $\nu_p((kp^2+i-1)_{(p-1)})\geq 2$, since $kp^2$ appears as a factor in $(kp^2+i-1)_{(p-1)}$. We know that $\nu_p(H_p((k-1)p^2+(p-1)p+i))\geq (k-1)(p-1)+(p-1)=k(p-1)$. Thus the summand $(kp^2+i-1)_{(p-1)}H_p((k-1)p^2+(p-1)p+i)$ has the $p$-adic valuation at least $k(p-1)+2$ while, by induction hypothesis $\nu_p(H_p(kp^2+i-1))=k(p-1)$. Hence, we get $\nu_p(H_p(kp^2+i))=k(p-1)$.
\end{proof}

Now we give the proofs of the main results of this section.

\begin{proof}[Proof of Theorem \ref{Hpperiod}]
Let us define $\gamma_n=\left\lfloor\frac{n}{p^2}\right\rfloor\cdot (p-1)+\left(\left\lfloor\frac{n}{p}\right\rfloor\pmod{p}\right)$ for simplicity of notation. At first, we show that
\begin{align}\label{period}
\frac{H_p(n+p^2)}{p^{\gamma_{n+p^2}}}\equiv -\frac{H_p(n)}{p^{\gamma_n}}\pmod{p}
\end{align}
for each $n\in\N$. We write $n=kp^2+jp+i$, $k\in\N$, $i,j\in\{0,1,...,p-1\}$, and proceed by induction on $jp+i$. By Theorem \ref{Hpmodp} we see that the congruence (\ref{period}) is satisfied for $jp+i=0$. If $jp+i>0$ and $i>0$ then
\begin{align*}
& \frac{H_p((k+1)p^2+jp+i)}{p^{\gamma_{(k+1)p^2+jp+i}}}=\frac{H_p((k+1)p^2+jp+i)}{p^{(k+1)(p-1)+j}}\\
& \equiv \frac{H_p((k+1)p^2+jp)}{p^{(k+1)(p-1)+j}}+\sum_{t=1}^j (-1)^t(j)_{(t)}\frac{H_p((k+1)p^2+(j-t)p)}{p^{(k+1)(p-1)+j-t}}\sum_{1\leq l_t\leq ...\leq l_1\leq i}\frac{1}{l_1\cdot ...\cdot l_t}\\
&\equiv -\frac{H_p(kp^2+jp)}{p^{k(p-1)+j}}-\sum_{t=1}^j (-1)^t(j)_{(t)}\frac{H_p(kp^2+(j-t)p)}{p^{k(p-1)+j-t}}\sum_{1\leq l_t\leq ...\leq l_1\leq i}\frac{1}{l_1\cdot ...\cdot l_t}\\
&\equiv -\frac{H_p(kp^2+jp+i)}{p^{k(p-1)+j}}=-\frac{H_p(kp^2+jp+i)}{p^{\gamma_{kp^2+jp+i}}}\pmod{p}.
\end{align*}
If $jp+i>0$ and $i=0$ then we perform similar computations as in the proof of Theorem \ref{Hpmodp}.
\begin{align*}
& \frac{H_p((k+1)p^2+jp)}{p^{\gamma_{(k+1)p^2+jp}}}=\frac{H_p((k+1)p^2+jp)}{p^{(k+1)(p-1)+j}}\\
& \equiv \frac{H_p((k+1)p^2+(j-1)p)}{p^{(k+1)(p-1)+j}}\\
& +\sum_{t=1}^{j-1} (-1)^t(j-1)_{(t)}\frac{H_p((k+1)p^2+(j-1-t)p)}{p^{(k+1)(p-1)+j-1-t}}\cdot\frac{1}{p}\sum_{1\leq l_t\leq ...\leq l_1\leq p-1}\frac{1}{l_1\cdot ...\cdot l_t}\\
& +\frac{(p-1)!H_p((k+1)p^2+(j-1)p)}{p^{(k+1)(p-1)+j}}\\
& \equiv \frac{1+(p-1)!}{p}\frac{H_p((k+1)p^2+(j-1)p)}{p^{(k+1)(p-1)+j-1}}\\
& +\sum_{t=1}^{j-1} (-1)^t(j-1)_{(t)}\frac{H_p((k+1)p^2+(j-1-t)p)}{p^{(k+1)(p-1)+j-1-t}}\cdot\frac{1}{p}\sum_{1\leq l_t\leq ...\leq l_1\leq p-1}\frac{1}{l_1\cdot ...\cdot l_t}\\
& \equiv -\frac{1+(p-1)!}{p}\frac{H_p(kp^2+(j-1)p)}{p^{k(p-1)+j-1}}\\
& -\sum_{t=1}^{j-1} (-1)^t(j-1)_{(t)}\frac{H_p(kp^2+(j-1-t)p)}{p^{k(p-1)+j-1-t}}\cdot\frac{1}{p}\sum_{1\leq l_t\leq ...\leq l_1\leq p-1}\frac{1}{l_1\cdot ...\cdot l_t}\\
&\equiv -\frac{H_p(kp^2+jp)}{p^{k(p-1)+j}}=-\frac{H_p(kp^2+jp)}{p^{\gamma_{kp^2+jp}}}\pmod{p}.
\end{align*}
We are left with proving that $p^2$ is a basic period of the sequence $\left((-1)^{\left\lfloor\frac{n}{p^2}\right\rfloor}\frac{H_p(n)}{p^{\gamma_n}}\pmod{p}\right)_{n\in\N}$. It suffices to show that $p$ is not a period of the sequence $\left((-1)^{\left\lfloor\frac{n}{p^2}\right\rfloor}\frac{H_p(n)}{p^{\gamma_n}}\pmod{p}\right)_{n\in\N}$. In order to do this, we assume the contrary. In particular, we have $\frac{H_p(p+1)}{p}\equiv H_p(1)=1\pmod{p}$. Using recurrence relation for numbers $H_p(n)$ we obtain the congruences
\begin{align*}
\frac{H_p(p)+(p)_{(p-1)}H_p(1)}{p}\equiv 1\pmod{p}
\end{align*}
or equivalently
\begin{align*}
\frac{H_p(p)}{p}+(p-1)_{(p-2)}H_p(1)\equiv 1\pmod{p}.
\end{align*}
Since $p$ is a period of the sequence $\left((-1)^{\left\lfloor\frac{n}{p^2}\right\rfloor}\frac{H_p(n)}{p^{\gamma_n}}\pmod{p}\right)_{n\in\N}$, we thus have $\frac{H_p(p)}{p}\equiv H_p(0)=1\pmod{p}$. Hence,
\begin{align*}
1+(p-1)_{(p-2)}H_p(1)\equiv 1\pmod{p}
\end{align*}
which means that
\begin{align*}
(p-1)_{(p-2)}\equiv 0\pmod{p}.
\end{align*}
However, the number $(p-1)_{(p-2)}$ is not divisible by $p$. A contradiction shows that $p$ is not a period of the sequence $\left((-1)^{\left\lfloor\frac{n}{p^2}\right\rfloor}\frac{H_p(n)}{p^{\gamma_n}}\pmod{p}\right)_{n\in\N}$.

The sequence $\left(\frac{H_p(n)}{p^{\gamma_n}}\pmod{p}\right)_{n\in\N}$ is thus periodic with period $2p^2$. We know, that $p^2$ is not a period of this sequence, as $\frac{H_p(kp^2)}{p^{k(p-1)}}\equiv (-1)^k\pmod{p}$. The number $2$ also cannot be a period of this sequence. If we assume the contrary then, since $H_p(0)\equiv H_p(1)\equiv 1\pmod{p}$, we have $\frac{H_p(n)}{p^{\gamma_n}}\equiv 1\pmod{p}$. This stays in contradiction with the fact, that $\frac{H_p(p^2)}{p^{p-1}}\equiv -1\pmod{p}$. It suffices to show that $2p$ is not a period of the sequence $\left(\frac{H_p(n)}{p^{\gamma_n}}\pmod{p}\right)_{n\in\N}$. In order to do this, we assume the contrary. In particular, we have $\frac{H_p(2p+1)}{p^2}\equiv H_p(1)=1\pmod{p}$ and $\frac{H_p(2p+2)}{p^2}\equiv H_p(2)=1\pmod{p}$. Using recurrence relation for numbers $H_p(n)$ we obtain the congruences
\begin{align*}
\frac{H_p(2p)+(2p)_{(p-1)}H_p(p+1)}{p^2} & \equiv 1\pmod{p}\\
\frac{H_p(2p+1)+(2p+1)_{(p-1)}H_p(p+2)}{p^2} & \equiv 1\pmod{p}
\end{align*}
or equivalently
\begin{align*}
\frac{H_p(2p)}{p^2}+\frac{2(p-1)_{(p-2)}H_p(p+1)}{p} & \equiv 1\pmod{p}\\
\frac{H_p(2p+1)}{p^2}+\frac{2(2p+1)(2p-1)_{(p-3)}H_p(p+2)}{p} & \equiv 1\pmod{p}.
\end{align*}
Since $2p$ is a period of the sequence $\left(\frac{H_p(n)}{p^{\gamma_n}}\pmod{p}\right)_{n\in\N}$, we thus have $\frac{H_p(2p)}{p^2}\equiv H_p(0)=1\pmod{p}$ and $\frac{H_p(2p+1)}{p^2}\equiv H_p(1)=1\pmod{p}$. Hence,
\begin{align*}
1+\frac{2(p-1)_{(p-2)}H_p(p+1)}{p} & \equiv 1\pmod{p}\\
1+\frac{2(2p+1)(2p-1)_{(p-3)}H_p(p+2)}{p} & \equiv 1\pmod{p},
\end{align*}
which means that
\begin{align*}
\frac{2(p-1)_{(p-2)}H_p(p+1)}{p} & \equiv 0\pmod{p}\\
\frac{2(2p+1)(2p-1)_{(p-3)}H_p(p+2)}{p} & \equiv 0\pmod{p}.
\end{align*}
The numbers $2(p-1)_{(p-2)}$ and $2(2p+1)(2p-1)_{(p-3)}$ are both not divisible by $p$. Hence, the numbers $H_p(p+1)$ and $H_p(p+2)$ are divisible by $p^2$. On the other hand, $H_p(p+2)=H_p(p+1)+(p+1)_{(p-1)}H_p(2)=H_p(p+1)+(p+1)p(p-1)_{(p-3)}$. The number $(p+1)p(p-1)_{(p-3)}$ is not divisible by $p^2$. This is why the numbers $H_p(p+1)$ and $H_p(p+2)$ cannot be both divisible by $p^2$. A contradiction shows that $p$ is not a period of the sequence $\left(\frac{H_p(n)}{p^{\gamma_n}}\pmod{p}\right)_{n\in\N}$.
\end{proof}

\begin{proof}[Proof of Proposition \ref{vpHpeq}]
If $p=2$, then the result is well known (see for example \cite[Theorem 4.1]{AmdMoll}). From Corollary \ref{vpHpnfrombelow} we know that $\nu_p(H_p(p-1))=0$. Hence we put $b_0=p-1$. Let us fix $j\in\{1,...,p-1\}$ and assume by induction hypothesis that we have $b_{j-1}\in\{p-j,...,p-1\}$ satisfies the equality $\nu_p(H_p((j-1)p+b_{j-1}))=j-1$. By Corollary \ref{vpHpnfrombelow} we have $\nu_p(H_p(jp+b_{j-1}-1))\geq j$ and $\nu_p(H_p(jp+b_{j-1}))\geq j$. Moreover, $\nu_p((jp+b_{j-1}-1)_{(p-1)})=1$. By the equality
$$H_p(jp+b_{j-1})=H_p(jp+b_{j-1}-1)+(jp+b_{j-1}-1)_{(p-1)}H_p((j-1)p+b_{j-1}),$$
since $\nu_p((jp+b_{j-1}-1)_{(p-1)}H_p((j-1)p+b_{j-1}))=j$, we infer that at least one of the numbers $H_p(jp+b_{j-1})$, $H_p(jp+b_{j-1}-1)$ has the $p$-adic valuation equal to $j$. We thus take $b_j$ as one of the values $b_{j-1}-1$, $b_{j-1}$ such that $\nu_p(H_p(jp+b_j))=j$. In particular, $b_j\geq p-j-1$. At last, we use Theorem \ref{Hpperiod} to obtain
$$(-1)^k\frac{H_p(kp^2+jp+b_j)}{p^{k(p-1)+j}}\equiv\frac{H_p(jp+b_j)}{p^j}\not\equiv 0\pmod{p},$$
since $\nu_p(H_p(jp+b_j))=j$. Hence $\nu_p(H_p(kp^2+jp+b_j))=k(p-1)+j$ for each $k\in\N$.
\end{proof}

In Theorem \ref{Hpperiod} we proved periodicity of the sequence $(H_{p}(n)/p^{\gamma_{p}(n)})_{n\in\N}$, where
$$
\gamma_{p}(n)=\left\lfloor\frac{n}{p^2}\right\rfloor\cdot (p-1)+\left(\left\lfloor\frac{n}{p}\right\rfloor\pmod{p}\right).
$$

In general, we ask the following.

\begin{ques}\label{Hpn/pvp}
For which $p\in\mathbb{P}$ and $w\in\N_+$ the sequence
$$
\left(\frac{H_{p}(n)}{p^{\nu_{p}(H_{p}(n))}}\pmod{p^w}\right)_{n\in\N}
$$
is periodic?
\end{ques}

\begin{thm}
For each prime number $p$ and positive integer $w$ the sequence $$\left(\frac{H_{p}(n)}{p^{\left\lfloor\frac{n}{p^2}\right\rfloor (p-1)}}\pmod{p^w}\right)_{n\in\N}$$ is periodic with period $p^{2w+1}(p-1)$.
\end{thm}

\begin{proof}
By \cite[Theorem B]{IOTY} we know that there exist power series $\lambda, f_r\in\Z_p[[x]]$, $r\in\{0,...,p^2-1\}$, convergent on $\Z_p$ such that $H_p(kp^2+r)=p^{k(p-1)}f_r(k)\prod_{l=1}^k\lambda(l)$ for each $k\in\N$ and $\nu_p(H_p(kp^2+r))=k(p-1)+\nu_p(f_r(k))$. Let us take any $n\in\N$ and write $n=kp^2+r$ for some $k\in\N$ and $r\in\{0,...,p^2-1\}$. Then, by the above fact we have the following chain of congruences for any $t\in\N$:
\begin{align*}
& \frac{H_{p}(n+tp^{2w+1}(p-1))}{p^{\left\lfloor\frac{n+tp^{2w+1}(p-1)}{p^2}\right\rfloor (p-1)}}=\frac{H_{p}((k+tp^{2w-1}(p-1))p^2+r)}{p^{(k+tp^{2w-1}(p-1))(p-1)}}\\
& =f_r(k+tp^{2w-1}(p-1))\prod_{l=1}^{k+tp^{2w-1}(p-1)}\lambda(l)\equiv f_r(k+tp^{2w-1}(p-1))\prod_{l=1}^{p^w}\lambda(l)^{\left\lceil\frac{k+tp^{2w-1}(p-1)-l+1}{p^w}\right\rceil}\\
& \equiv f_r(k)\prod_{l=1}^{p^w}\lambda(l)^{\left\lceil\frac{k-l+1}{p^w}\right\rceil+tp^{w-1}(p-1)}\equiv f_r(k)\prod_{l=1}^{p^w}\lambda(l)^{\left\lceil\frac{k-l+1}{p^w}\right\rceil}\equiv f_r(k)\prod_{l=1}^{k}\lambda(l)=\frac{H_p(kp^2+r)}{p^{k(p-1)}}\\
& =\frac{H_{p}(n)}{p^{\left\lfloor\frac{n}{p^2}\right\rfloor (p-1)}}\pmod{p^w},
\end{align*}
where we use Euler's theorem as $p\nmid\lambda(l)$ for any $l\in\N$.
\end{proof}

Let us define a particular subset of the set of prime numbers:
\begin{align*}
\cal{A}=\{p\in\mathbb{P}: \exists_{\alpha_p(p),...,\alpha_{p^2-1}(p)\in\N_+}\forall_{r\in\{p,...,p^2-1\}}\forall_{k\in\N} \nu_p(H_p(kp^2+r))=k(p-1)+\alpha_r(p)\}.
\end{align*}

Then we are able to give a partial answer to Question \ref{Hpn/pvp}.

\begin{cor}
If $p\in\cal{A}$ then for each $w\in\N_+$ the sequence $\left(\frac{H_{p}(n)}{p^{\nu_{p}(H_{p}(n))}}\pmod{p^w}\right)_{n\in\N}$ is periodic with period $p^{2w+2\alpha(p)+1}(p-1)$, where $\alpha(p)=\max_{r\in\{p,...,p^2-1\}} \alpha_r(p)$
\end{cor}

\begin{proof}
By the previous theorem we have
$$\frac{H_{p}(n+tp^{2w+2\alpha(p)+1}(p-1))}{p^{\left\lfloor\frac{n+tp^{2w+2\alpha(p)+1}(p-1)}{p^2}\right\rfloor (p-1)}}\equiv \frac{H_{p}(n)}{p^{\left\lfloor\frac{n}{p^2}\right\rfloor (p-1)}}\pmod{p^{w+\alpha(p)}}$$
for each $t\in\N$. After division by $p^{\alpha_r(p)}$, where $r=n\pmod{p^2}$, we get
$$\frac{H_{p}(n+tp^{2w+2\alpha(p)+2})}{p^{\nu_{p}(H_{p}(n+tp^{2w+2\alpha(p)+2}))}}\equiv \frac{H_{p}(n)}{p^{\nu_{p}(H_{p}(n))}}\pmod{p^{w+\alpha(p)-\alpha_r(p)}}.$$
Since $\alpha(p)-\alpha_r(p)\geq 0$, we thus obtain
$$\frac{H_{p}(n+tp^{2w+2\alpha(p)+2})}{p^{\nu_{p}(H_{p}(n+tp^{2w+2\alpha(p)+2}))}}\equiv \frac{H_{p}(n)}{p^{\nu_{p}(H_{p}(n))}}\pmod{p^w},$$
as we wanted to prove.
\end{proof}

\begin{cor}
The sequence
$$
\left(\frac{H_{2}(n)}{2^{\nu_{2}(H_{2}(n))}}\pmod{4}\right)_{n\in\N}
$$
is periodic of period $16$.
\end{cor}

\begin{proof}
Since $2\in\cal{A}$ and $\alpha(2)=2$ (see for example \cite[Theorem 4.1]{AmdMoll}), by the previous corollary we know that the sequence $\left(\frac{H_{2}(n)}{2^{\nu_{2}(H_{2}(n))}}\pmod{4}\right)_{n\in\N}$ is periodic with period $2^9$. Hence it suffices to check the values $\frac{H_{2}(n)}{2^{\nu_{2}(H_{2}(n))}}$ modulo $4$ for $n\in\{0,...,511\}$ one by one.
\end{proof}

In case of $p\not\in\cal{A}$ we predict that the sequence $\left(\frac{H_{p}(n)}{p^{\nu_{p}(H_{p}(n))}}\pmod{p^w}\right)_{n\in\N}$ is not periodic for any $w\in\N_+$ but proving this fact seems to be out of reach for us.

\section{The behaviour of the $p$-adic valuation of $H_d(n)$, where $p>d$}\label{Section5}

On the beginning of the paper we marked that if $p$ is a prime number $<d$, then $\nu_p(H_d(n))=0$ for each $n\in\N$. The problem of computation of the  $p$-adic valuations of the numbers $H_d(n)$ in case $p=d$ was considered in the previous section and it was studied in several publications. In the opposition to this fact the question concerning the computation of the $p$-adic valuations of the numbers $H_d(n)$, where $p>d$, is almost unexplored. It was first considered by Amdeberhan and Moll in \cite{AmdMoll}. According to results on periodicity of the sequence $(H_2(n)\pmod{c})_{n\in\N}$ they claimed that if $p$ is an odd prime number not dividing the values $H_2(n)$ for $n\in\{0,...,p-1\}$, then $\nu_p(H_d(n))=0$ for each $n\in\N$. The situation is more complicated if $p$ divides some of those values. It happens for $p=5$. Based on numerical computations they stated a conjecture which equivalent version is as follows

\begin{conj}\label{vpH2n}
If $p$ is an odd prime number such that $p\mid H_2(n)$ for some $n\in\{0,...p-1\}$ then for each $k\in\N_+$ there exists a unique $n_k$ modulo $p^k$ such that $p^k\mid H_2(n)$ if and only if $n\equiv n_k\pmod{p^k}$.
\end{conj}

The above conjecture is not true for two reasons. The first reason is that there exist prime numbers $p$ and positive integers $n_1$ such that $\nu_p(H_2(n))=1$ for each $n\equiv n_1\pmod{p}$, for example $(p,n_1)=(19,6)$. The second one is that for some prime numbers $p$, for example $59$ and $61$, there exist more than one value $n_1$ modulo $p$ such that $p\mid H_2(n)$ for $n\equiv n_1\pmod{p}$.

The aim of this section is to shed some light on the behaviour of $\nu_{p}(H_{d}(n))$ in the case when $p>d$.

We use the exact formula for the number $H_d(n)$ and make some modifications:
\begin{align*}
H_d(n)=\sum_{k=0}^{\left\lfloor\frac{n}{d}\right\rfloor} \frac{n!}{(n-dk)!k!d^k}=\sum_{k=0}^{\left\lfloor\frac{n}{d}\right\rfloor} \frac{(n)_{(dk)}}{k!d^k}=\sum_{k=0}^{+\infty} \frac{(n)_{(dk)}}{k!d^k}.
\end{align*}
The above notation suggests us to define the function $f_d(x)=\sum_{k=0}^{+\infty} \frac{(x)_{(dk)}}{k!d^k}$, $x\in\Z_p$. This function is well defined for each $x\in\Z_p$ as
\begin{align*}
\nu_p\left(\frac{(x)_{(dk)}}{k!d^k}\right)\geq \nu_p\left(\frac{(dk)!}{k!d^k}\right)=\nu_p\left(\frac{(dk)!}{k!}\right)=\nu_p((dk)_{((d-1)k)})\geq \nu_p(((d-1)k)!),
\end{align*}
where the last value in the above inequalities tend to $+\infty$ when $k\rightarrow +\infty$. The following three lemmas show that
\begin{itemize}
\item $f_d$ is a $p$-adic differentiable function,
\item there exists a function $g_d:\Z_p\times p\Z_p\rightarrow\Q_p$ such that $f_d(x+h)=f_d(x)+hf'_d(x)+h^2g_d(x,h)$ for any $x\in\Z_p$, $h\in p\Z_p$ and $g_d(x,h)\in\Z_p$ if only $d\geq 3$ or $p\mid f_2(x)$,
\item $f'_d(x_1)\equiv f'_d(x_2)\pmod{p^r}$ if $x_1\equiv x_2\pmod{p^r}$.
\end{itemize}
By the above facts we will be able to describe behaviour of the $p$-adic valuations of numbers $H_d(n)$, $n\in\N$, using classical Hensel's lemma.

\begin{lem}
The function $f_d$ is differentiable and $f'_d(x)=\sum_{k=1}^{+\infty}\sum_{j=0}^{dk-1} \frac{(x)_{(j)}(x-j-1)_{(dk-j-1)}}{k!d^k}\in\Z_p$ for each $x\in\Z_p$.
\end{lem}

\begin{proof}
We compute the derivative from the very definition:
\begin{align*}
& f'_d(x)=\lim_{h\rightarrow 0}\frac{f_d(x+h)-f_d(x)}{h}=\lim_{h\rightarrow 0}\frac{1}{h}\left(\sum_{k=0}^{+\infty} \frac{(x+h)_{(dk)}}{k!d^k}-\sum_{k=0}^{+\infty} \frac{(x)_{(dk)}}{k!d^k}\right)\\
& =\lim_{h\rightarrow 0}\frac{1}{h}\sum_{k=0}^{+\infty} \frac{(x+h)_{(dk)}-(x)_{(dk)}}{k!d^k}\\
& =\lim_{h\rightarrow 0}\frac{1}{h}\sum_{k=0}^{+\infty} \frac{1}{k!d^k}\left[\left(\sum_{t=0}^{dk}h^t\sum_{0\leq j_1<...<j_t\leq dk-1}\frac{(x)_{(dk)}}{(x-j_1)\cdot ...\cdot(x-j_t)}\right)-(x)_{(dk)}\right]\\
& =\lim_{h\rightarrow 0}\frac{1}{h}\sum_{k=1}^{+\infty} \frac{1}{k!d^k}\sum_{t=1}^{dk}h^t\sum_{0\leq j_1<...<j_t\leq dk-1}\frac{(x)_{(dk)}}{(x-j_1)\cdot ...\cdot(x-j_t)}\\
& =\lim_{h\rightarrow 0}\sum_{k=1}^{+\infty} \frac{1}{k!d^k}\sum_{t=1}^{dk}h^{t-1}\sum_{0\leq j_1<...<j_t\leq dk-1}\frac{(x)_{(dk)}}{(x-j_1)\cdot ...\cdot(x-j_t)}\\
& =\lim_{h\rightarrow 0}\sum_{k=1}^{+\infty} \frac{1}{k!d^k}\left(\sum_{j=0}^{dk-1}\frac{(x)_{(dk)}}{x-j}+h\sum_{t=2}^{dk}h^{t-2}\sum_{0\leq j_1<...<j_t\leq dk-1}\frac{(x)_{(dk)}}{(x-j_1)\cdot ...\cdot(x-j_t)}\right)\\
& =\lim_{h\rightarrow 0}\sum_{k=1}^{+\infty} \sum_{j=0}^{dk-1}\frac{(x)_{(dk)}}{k!d^k(x-j)}+\sum_{k=1}^{+\infty} h\sum_{t=2}^{dk}h^{t-2}\sum_{0\leq j_1<...<j_t\leq dk-1}\frac{(x)_{(dk)}}{k!d^k(x-j_1)\cdot ...\cdot(x-j_t)},
\end{align*}
where we put $\sum_{0\leq j_1<...<j_t\leq dk-1}\frac{(x)_{(dk)}}{(x-j_1)\cdot ...\cdot(x-j_t)}=(x)_{(dk)}$ for $t=0$. We estimate the $p$-adic valuation of $\frac{h^{t-2}(x)_{(dk)}}{k!d^k(x-j_1)\cdot ...\cdot(x-j_t)}$ from below. We will do this by estimating the $p$-adic valuation of the product $(x-j_1)\cdot ...\cdot (x-j_{dk-l})$, where $0\leq j_1<...<j_{dk-l}\leq dk-1$, from below by the number of factors $x-j_1, ..., x-j_{dk-l}$ divisible by $p$. This number is at least equal to $\left\lfloor\frac{dk}{p}\right\rfloor -l$. Indeed, we have at least $\left\lfloor\frac{dk}{p}\right\rfloor$ numbers divisible by $p$ among all the $p$-adic integers $x-dk+1, x-dk+2, ..., x$ and we cancel at most $l$ numbers divisible by $p$ from them. Hence,
\begin{align*}
\nu_p &\left(\frac{h^{t-2}(x)_{(dk)}}{k!d^k(x-j_1)\cdot ...\cdot(x-j_t)}\right)\geq (t-2)\nu_p(h)+\left\lfloor\frac{dk}{p}\right\rfloor-t-\frac{k-s_p(k)}{p-1}\\
& >t-2+\frac{dk}{p}-1-t-\frac{k}{p-1}=\frac{(p-1)dk-pk}{p(p-1)}-3=\frac{[(p-1)(d-1)-1]k}{p(p-1)}-3\geq\frac{[2\cdot 1-1]k}{p(p-1)}-3\\
& =\frac{k}{p(p-1)}-3,
\end{align*}
where we assumed that $\nu_p(h)\geq 1$ and use the Legendre formula $\nu_p(k!)=\frac{k-s_p(k)}{p-1}$. Here, $s_p(k)$ is the sum of digits of the number $k$ written in the (unique) $p$-ary expansion. Hence, the $p$-adic valuation of $$\sum_{t=2}^{dk}h^{t-2}\sum_{0\leq j_1<...<j_t\leq dk-1}\frac{(x)_{(dk)}}{k!d^k(x-j_1)\cdot ...\cdot(x-j_t)}$$ is at least $-2$ and it tends to $+\infty$ when $k\rightarrow +\infty$. Thus, the series $$\sum_{k=1}^{+\infty} \sum_{t=2}^{dk}h^{t-2}\sum_{0\leq j_1<...<j_t\leq dk-1}\frac{(x)_{(dk)}}{k!d^k(x-j_1)\cdot ...\cdot(x-j_t)}$$ is convergent and its $p$-adic valuation is at least $-2$. We infer that
\begin{align*}
& \nu_p\left(\sum_{k=1}^{+\infty} h\sum_{t=2}^{dk}h^{t-2}\sum_{0\leq j_1<...<j_t\leq dk-1}\frac{(x)_{(dk)}}{k!d^k(x-j_1)\cdot ...\cdot(x-j_t)}\right)\\
& =\nu_p(h)+\nu_p\left(\sum_{k=2}^{+\infty} \sum_{t=2}^{dk}h^{t-2}\sum_{0\leq j_1<...<j_t\leq dk-1}\frac{(x)_{(dk)}}{k!d^k(x-j_1)\cdot ...\cdot(x-j_t)}\right)\rightarrow +\infty
\end{align*}
when $h\rightarrow 0$. This is why $f'_d(x)=\sum_{k=1}^{+\infty} \sum_{j=0}^{dk-1}\frac{(x)_{(dk)}}{k!d^k(x-j)}=\sum_{k=1}^{+\infty}\sum_{j=0}^{dk-1} \frac{(x)_{(j)}(x-j-1)_{(dk-j-1)}}{k!d^k}$. This series is convergent, as
\begin{align*}
& \nu_p\left(\frac{(x)_{(j)}(x-j-1)_{(dk-j-1)}}{k!d^k}\right)=\nu_p\left({x\choose j}{x-j-1\choose dk-j-1}\frac{j!(dk-j-1)!}{k!d^k}\right)\\
& \geq\frac{j-s_p(j)+dk-j-1-s_p(dk-j-1)-k+s_p(k)}{p-1}\geq\frac{(d-1)k-1}{p-1}-2\log_p dk-2\rightarrow +\infty
\end{align*}
when $k\rightarrow +\infty$, where we used the fact that $j,dk-j-1<dk$, so $j,dk-j-1$ have at most $\lfloor\log_p dk\rfloor+1$ $p$-ary digits. Moreover,
\begin{align*}
& \nu_p\left(\frac{(x)_{(j)}(x-j-1)_{(dk-j-1)}}{k!d^k}\right)=\nu_p\left({x\choose j}{x-j-1\choose dk-j-1}\frac{j!(dk-j-1)!}{k!d^k}\right)\geq 0,
\end{align*}
as at least one of the numbers $j,dk-j-1$ is greater than or equal to $k$. Thus the $p$-adic valuation of $\sum_{k=1}^{+\infty}\sum_{j=0}^{dk-1} \frac{(x)_{(j)}(x-j-1)_{(dk-j-1)}}{k!d^k}$ is non-negative.
\end{proof}

\begin{lem}
For each $x,h\in\Z_p$ with $\nu_p(h)\geq 1$ we have $f_d(x+h)=f_d(x)+hf'_d(x)+h^2g_d(x,h)$, where $g_d(x,h)=\sum_{l=2}^{+\infty}x^{l-2}\sum_{k=\left\lceil\frac{l}{d}\right\rceil}^{+\infty}\frac{1}{k!d^k}\sum_{0\leq j_1<...<j_{dk-l}\leq dk-1}(x-j_1)\cdot ...\cdot (x-j_{dk-l})$. Moreover, if $d\geq 3$ or $\nu_p(f_2(x))\geq 1$, then $g_d(x,h)\in\Z_p$.
\end{lem}

\begin{proof}
Let us fix $x,h\in\Z_p$ with $\nu_p(h)\geq 1$. We then have:
\begin{equation}\label{fd(x)}
\begin{split}
& f_d(x+h)=\sum_{k=0}^{+\infty} \frac{(x+h)_{(dk)}}{k!d^k}=\sum_{k=0}^{+\infty}\frac{1}{k!d^k}\sum_{l=0}^{dk}h^l\sum_{0\leq j_1<...<j_{dk-l}\leq dk-1}(x-j_1)\cdot ...\cdot (x-j_{dk-l}).
\end{split}
\end{equation}
For the change of the order of summation, we will show that $$\lim_{(k,l)\rightarrow +\infty} \nu_p\left(\frac{h^l\sum_{0\leq j_1<...<j_{dk-l}\leq dk-1}(x-j_1)\cdot ...\cdot (x-j_{dk-l})}{k!d^k}\right)=+\infty.$$ Let us estimate the $p$-adic valuation of the number $\sum_{0\leq j_1<...<j_{dk-l}\leq dk-1}(x-j_1)\cdot ...\cdot (x-j_{dk-l})$ from below. We have the inequality
\begin{equation}\label{summandvpestim}
\begin{split}
& \nu_p\left(\frac{h^l\sum_{0\leq j_1<...<j_{dk-l}\leq dk-1}(x-j_1)\cdot ...\cdot (x-j_{dk-l})}{k!d^k}\right)\geq l+\left\lfloor\frac{dk}{p}\right\rfloor -l-\frac{k-s_p(k)}{p-1}\\
& =\left\lfloor\frac{dk}{p}\right\rfloor -\frac{k-s_p(k)}{p-1}=\left\lfloor\frac{dk}{p} -\frac{k-s_p(k)}{p-1}\right\rfloor=\left\lfloor\frac{(dp-d-p)k+ps_p(k)}{p(p-1)}\right\rfloor\\
& =\left\lfloor\frac{[(d-1)(p-1)-1]k+ps_p(k)}{p(p-1)}\right\rfloor.
\end{split}
\end{equation}
Since $p>d\geq 2$, we thus have
$$\lim_{(k,l)\rightarrow +\infty} \nu_p\left(\frac{h^l\sum_{0\leq j_1<...<j_{dk-l}\leq dk-1}(x-j_1)\cdot ...\cdot (x-j_{dk-l})}{k!d^k}\right)=+\infty.$$
Then we change summation in the last expression in (\ref{fd(x)}) and get the following:
\begin{align*}
& f_d(x+h)=\sum_{l=0}^{+\infty}h^l\sum_{k=\left\lceil\frac{l}{d}\right\rceil}^{+\infty}\frac{1}{k!d^k}\sum_{0\leq j_1<...<j_{dk-l}\leq dk-1}(x-j_1)\cdot ...\cdot (x-j_{dk-l})\\
& =f_d(x)+hf'_d(x)+h^2g_d(x,h),
\end{align*}
where $g_d(x,h)=\sum_{l=2}^{+\infty}h^{l-2}\sum_{k=\left\lceil\frac{l}{d}\right\rceil}^{+\infty}\frac{1}{k!d^k}\sum_{0\leq j_1<...<j_{dk-l}\leq dk-1}(x-j_1)\cdot ...\cdot (x-j_{dk-l})$. Now, we are left with the proof that $\nu_p(g_d(x,h))\geq 0$ provided that $\nu_p(f_d(x))\geq 1$ or $d\geq 3$. If $k<p$, where $d\geq 2$ is arbitrary, then, of course, $\nu_p\left(\frac{h^{l-2}\sum_{0\leq j_1<...<j_{dk-l}\leq dk-1}(x-j_1)\cdot ...\cdot (x-j_{dk-l})}{k!d^k}\right)\geq 0$. By (\ref{summandvpestim}) we know that
$$\nu_p\left(\frac{h^{l-2}\sum_{0\leq j_1<...<j_{dk-l}\leq dk-1}(x-j_1)\cdot ...\cdot (x-j_{dk-l})}{k!d^k}\right)\geq \left\lfloor\frac{[(d-1)(p-1)-1]k+ps_p(k)}{p(p-1)}\right\rfloor -2.$$
We check that $\frac{[(d-1)(p-1)-1]k+ps_p(k)}{p(p-1)}-2\geq\frac{[2(p-1)-1](p+1)}{p(p-1)}-2=\frac{(2p-3)(p+1)}{p(p-1)}-2\geq 0$ for $p>d\geq 3$ and $k\geq p+1$. If $d\geq 3$ and $k=p$, then either $l\geq 3$ and then $\nu_p(h^{l-2})\geq l-2\geq 1=\nu_p(p!)$ or $l=2$ an then in each product of the form $(x-j_1)\cdot ...\cdot (x-j_{dk-2})$, where $0\leq j_1<...<j_{dk-2}\leq dk-1$ there is a factor with the $p$-adic valuation at least equal to $1$. Hence, $\nu_p\left(\frac{h^{l-2}\sum_{0\leq j_1<...<j_{dk-l}\leq dk-1}(x-j_1)\cdot ...\cdot (x-j_{dk-l})}{k!d^k}\right)\geq 0$ for $d\geq 3$ and $k=p$. If $d=2$, then $\frac{[(d-1)(p-1)-1]k+ps_p(k)}{p(p-1)}-2\geq\frac{[(p-1)-1](2p+4)}{p(p-1)}-2=\frac{(p-2)(2p+4)}{p(p-1)}-2\geq 0$ for $p\geq 5$ and $k\geq 2p+4$ (we do not need to consider the case of $p=3$ as $3\nmid H_2(n)$ for any $n\in\N$). If $d=2$, $p\leq k\leq 2p-1$ and $l\geq 3$, then $\nu_p(h^{l-2})\geq l-2\geq 1=\nu_p(k!)$. Hence, $\nu_p\left(\frac{h^{l-2}\sum_{0\leq j_1<...<j_{dk-l}\leq dk-1}(x-j_1)\cdot ...\cdot (x-j_{dk-l})}{k!d^k}\right)\geq 0$ for $d=2$, $p\leq k\leq 2p-1$ and $l\geq 3$. If $d=2$ and $2p\leq k\leq 2p+3$, where $p\geq 5$, then either $l\geq 4$ and then $\nu_p(h^{l-2})\geq l-2\geq 2=\nu_p(k!)$ or $l\in\{2,3\}$ and then in each product of the form $(x-j_1)\cdot ...\cdot (x-j_{2k-l})$, where $0\leq j_1<...<j_{2k-l}\leq 2k-1$ there are at least $4-l$ factors with the $p$-adic valuation at least equal to $1$. Hence, $\nu_p(h^{l-2}\sum_{0\leq j_1<...<j_{2k-l}\leq dk-1}(x-j_1)\cdot ...\cdot (x-j_{dk-l}))\geq l-2+4-l=2=\nu_p(k!)$. We left with the case of $d=2$, $p\leq k\leq 2p-1$ and $l=2$. Let us assume that $\nu_p(f_2(x))\geq 1$ and write $r=x\pmod{p}$. We then compute modulo $p$:
\begin{align*}
& p\cdot \sum_{k=p}^{2p-1}\sum_{0\leq j_1<...<j_{2k-2}\leq 2k-1} \frac{(x-j_1)\cdot ...\cdot (x-j_{2k-l})}{k!2^k}\\
& =\sum_{k=p}^{2p-1}\sum_{0\leq j_1<...<j_{2k-2}\leq 2k-1} \frac{(x-j_1)\cdot ...\cdot (x-j_{2k-l})}{(k)_{(k-p)}(p-1)!2^k}\\
&\equiv -\sum_{k=p}^{2p-1} \frac{(x)_{(r)}(x-r-1)_{(p-1)}(x-r-p-1)_{(p-r-1)}(x-2p)_{(2(k-p))}}{(k-p)_{(k-p)}2^{k-p+1}}\\
& \equiv -\frac{1}{2}\sum_{k=p}^{2p-1} \frac{r!(p-1)!(p-1)_{(p-r-1)}(x)_{(2(k-p))}}{(k-p)_{(k-p)}2^{k-p}}=-\frac{1}{2}\sum_{k=0}^{p-1} \frac{((p-1)!)^2(x)_{(2k)}}{k!2^k}\\
& \equiv \sum_{k=0}^{+\infty} \frac{(x)_{(2k)}}{k!2^k}=-\frac{1}{2}f_2(x)\equiv 0\pmod{p}.
\end{align*}
Thus,
$$
\nu_p\left(\sum_{k=p}^{2p-1}\sum_{0\leq j_1<...<j_{2k-2}\leq 2k-1} \frac{(x-j_1)\cdot ...\cdot (x-j_{2k-l})}{k!2^k}\right)\geq 0.
$$
This ends the proof of the fact that $g_d(x,h)\in\Z_p$ on condition that $\nu_p(f_d(x))\geq 1$ or $d\geq 3$.
\end{proof}

\begin{lem}
Let $x,h\in\Z_p$ with $\nu_p(h)\geq 1$. If $d\geq 3$ or $\nu_p(f_2(x))\geq 1$, then $f'_d(x+h)\equiv f'_d(x)\pmod{p^{\nu_p(h)}}$.
\end{lem}

\begin{proof}
By the previous lemma we can write
\begin{align*}
f_d(x+h)= & f_d(x)+hf'_d(x)+h^2g_d(x,h),\\
f_d(x)= & f_d(x+h)-hf'_d(x+h)+h^2g_d(x+h,-h).
\end{align*}
After adding the above two equalities and simplifying we obtain
$$0=h(f'_d(x)-f'_d(x+h))+h^2(g_d(x,h)+g_d(x+h,-h)).$$
Since $d\geq 3$ or $\nu_p(f_2(x))\geq 1$, by the previous lemma we have that $\nu_p(g_d(x,h)),\nu_p(g_d(x+h,-h))\geq 0$. Thus, $p^{2\nu_p(h)}\mid h^2(g_d(x,h)+g_d(x+h,-h))$, which implies that $p^{2\nu_p(h)}\mid h(f'_d(x)-f'_d(x+h))$. Hence, $p^{\nu_p(h)}\mid f'_d(x)-f'_d(x+h)$, or equivalently, $f'_d(x+h)\equiv f'_d(x)\pmod{p^{\nu_p(h)}}$.
\end{proof}

At this moment we are prepared to state the theorem on qualitative description of the sequence $(\nu_p(H_d(n))_{n\in\N}$.

\begin{thm}
Let $d$ be a positive integer at least equal to $2$, $p>d$ be a prime number and $k$ be a positive integer. Assume that $H_d(n_k)\equiv 0 \pmod{p^k}$ for some integer $n_k$. Then the number of solutions $n$  modulo $p^{k+1}$ of the congruence $H_d(n)\equiv 0 \pmod{p^{k+1}}$, satisfying the condition $n \equiv n_k \pmod{p^k}$, is equal to:
\begin{itemize}
\item $1$, when $f'_d(n_k)\not\equiv 0 \pmod{p}$;
\item $0$, when $f'_d(n_k)\equiv 0 \pmod{p}$ and $H_d(n_k)\not\equiv 0 \pmod{p^{k+1}}$;
\item $p$, when $f'_d(n_k)\equiv 0 \pmod{p}$ and $H_d(n_k)\equiv 0 \pmod{p^{k+1}}$.
\end{itemize}
In particular, if $p\mid H_d(n_1)$ for some $n_1\in\N$ and $\nu_p(f'_d(n_1))=0$, then for each $k\in\N_+$ there exists a unique value $n_k$ modulo $p^k$ such that if $n\equiv n_1\pmod{p}$, then $p^k\mid H_d(n)$ if and only if $n\equiv n_k\pmod{p^k}$. In other words there exists a unique $p$-adic integer $b\equiv n_1\pmod{p}$ such that $\nu_p(H_d(n))=\nu_p(n-b)$ for each non-negative integer $n\equiv n_1\pmod{p}$.
\end{thm}

\begin{proof}
We have $H_d(n)=f_d(n)$ for each $n\in\N$. Moreover, we know that $f_d(x+h)=f_d(x)+hf'_d(x)+h^2g_d(x,h)$ for $x,h\in\Z_p$ with $\nu_p(h)\geq 1$. Furthermore, if $d\geq 3$ or $\nu_p(f_2(x))\geq 1$, then $g_d(x,h)\in\Z_p$ and $f'_d(x_1)\equiv f'_d(x_2)\pmod{p}$ for congruent $p$-adic integers $x_1, x_2$ modulo $p$. These all facts make the proof of our theorem exactly the same as the proof of classical Hensel's lemma from \cite[p. 44]{Nar}.
\end{proof}

\begin{rem}
{\rm Let us compute $f'_d(n)\pmod{p}$ for $n\in\N$:
\begin{align*}
f'_d(n)= & \sum_{k=1}^{+\infty}\sum_{j=0}^{dk-1} \frac{(n)_{(j)}(n-j-1)_{(dk-j-1)}}{k!d^k}=\sum_{k=1}^{\left\lfloor\frac{n}{d}\right\rfloor}\sum_{j=0}^{dk-1} \frac{(n)_{(j)}(n-j-1)_{(dk-j-1)}}{k!d^k}\\
&\quad +\sum_{k=\left\lfloor\frac{n}{d}\right\rfloor +1}^{+\infty} \frac{n!(-1)^{dk-n-1}(dk-n-1)!}{k!d^k}\equiv \sum_{k=1}^{\left\lfloor\frac{n}{d}\right\rfloor}\sum_{j=0}^{dk-1} \frac{(n)_{(j)}(n-j-1)_{(dk-j-1)}}{k!d^k}\\
&\quad +\sum_{k=\left\lfloor\frac{n}{d}\right\rfloor +1}^{\left\lceil\frac{n+p+1}{d-1}\right\rceil -1} \frac{n!(-1)^{dk-n-1}(dk-n-1)!}{k!d^k}\pmod{p},
\end{align*}
as $\frac{(dk-n-1)!}{k!}={dk-n-1\choose k}((d-1)k-n-1)!$ and $p\mid ((d-1)k-n-1)!$ if $p\leq (d-1)k-n-1$. That is a way to compute the value $f'_d(n)\pmod{p}$ by adding only finitely many summands.

However, in many cases, having $d\in\N_{\geq 2}$, a prime number $p$ and $n_1\in\{d,...,p-1\}$ such that $p\mid H_d(n_1)$ (recall that $H_d(n)=1$ for $n\in\{0,...,d-1\}$), we do not need to compute $f'_d(n)\pmod{p}$. Instead of this, it suffices  to compute $\nu_p(H_d(n_1))$. If $\nu_p(H_d(n_1))=1$, then we check $p$-adic valuations of numbers $H_d(n_1+pt)$ for $t\in\{1,...,p-1\}$. If $\nu_p(H_d(n_1+pt))\geq 2$ for some $t\in\{1,...,p-1\}$ then for each $k\in\N_+$ there exists a unique value $n_k$ modulo $p^k$ such that if $n\equiv n_1\pmod{p}$ then $p^k\mid H_d(n)$ if and only if $n\equiv n_k\pmod{p^k}$. Otherwise, $\nu_p(H_d(n))=1$ for any $n\equiv n_1\pmod{p}$. If $\nu_p(H_d(n_1))\geq 2$ then we only check $p$-adic valuation of the number $H_d(n_1+p)$. If $\nu_p(H_d(n_1+p))=1$ then for each $k\in\N_+$ there exists a unique value $n_k$ modulo $p^k$ such that if $n\equiv n_1\pmod{p}$ then $p^k\mid H_d(n)$ if and only if $n\equiv n_k\pmod{p^k}$. Otherwise, $\nu_p(H_d(n))\geq 2$ for any $n\equiv n_1\pmod{p}$.}
\end{rem}

\begin{rem}
{\rm We found all the triples $(d,p,n_1)$ for $d\in\{2,3,4,5\}$, first $25$ prime numbers $p$ and numbers $n_1\in\{d,...,p-1\}$ such that $p\mid H_d(n_1)$. We checked the behaviour of the $p$-adic valuations of numbers $H_d(n)$ where $n\equiv n_1\pmod{p}$. The only triples $(d,p,n_1)$ such that $\nu_p(H_d(n))=1$ for all $n\equiv n_1\pmod{p}$ are $(2,19,6)$, $(3,13,7)$. For any other triple and a positive integer $k$ we have the existence of a unique $n_k$ modulo $p^k$ such that if $n\equiv n_1\pmod{p}$ then $p^k\mid H_d(n)$ if and only if $n\equiv n_k\pmod{p^k}$.}
\end{rem}

\section{Properties of polynomials attained during differentiation the exponential generating function of the sequence $(H_d(n))_{n\in\N}$}\label{Section6}

Let $m\in\N$ and for $m$-th derivative of $\cal{H}_d(x)$ we write $\cal{H}_{d}^{(m)}(x)=W_{d,m}(x)\cal{H}_{d}(x)$. We have $W_{d,0}(x)=1$ and due to the identity
$$
\cal{H}_{d}^{(m+1)}(x)=(\cal{H}_{d}^{(m)}(x))'=W_{d,m}'(x)\cal{H}_{d}(x)+(1+x^{d-1})W_{d,m}(x)\cal{H}_{d}(x)
$$
we immediately deduce the recurrence relation
\begin{equation}\label{Wdmpol}
W_{d,m+1}(x)=W_{d,m}'(x)+(1+x^{d-1})W_{d,m}(x).
\end{equation}

We summarize the basic properties of the sequence $(W_{d,m}(x))_{m\in\N}$ in the following:

\begin{thm}\label{Wdn}
The following properties hold:
\begin{enumerate}
\item $W_{d,n}(x)$ is unitary polynomial with non-negative integer coefficients and $\op{deg}W_{d,n}(x)=(d-1)n$;
\item $W_{d,n}(0)=H_{d}(n)$;
\item For $d\in\N_{\geq 2}$ and $m\in\N$ let us write
$$
W_{d,m}(x)=\sum_{i=0}^{(d-1)m}A_{d}(i,m)x^{i}.
$$
If $d\in\N_{\geq 3}$ is an odd number then
$$
\frac{1}{n!}\sum_{k=0}^{n}\binom{n}{k}(-1)^{n-k}H_{d}(m+k)H_{d}(n-k)=
\begin{cases}
\begin{array}{lll}
  A_{d}(n,m) &  &\;\mbox{for}\;n\leq (d-1)m, \\
  0          &  &\;\mbox{for}\;n>(d-1)m
\end{array}
\end{cases};
$$
\item The exponential generating function for the sequence $(W_{d,n}(x))_{n\in\N}$ takes the form
$$
\cal{W}_{d}(x,t)=\sum_{n=0}^{\infty}\frac{W_{d,n}(x)}{n!}t^{n}=\cal{H}_{d}(t)e^{\frac{(t+x)^d-t^d-x^d}{d}}=e^{t+\frac{(t+x)^{d}-x^{d}}{d}}.
$$
\end{enumerate}
\end{thm}
\begin{proof}
\noindent (1) This is clear. Because $W_{d,0}(x)=1$ is a polynomial and on the right hand side of the recurrence $W_{d,m}(x)=W_{d,m-1}'(x)+(1+x^{d-1})W_{d,m-1}(x)$, $m\in\N_{+}$, we have only addition and multiplication then simple induction gives the result. Similarly simple induction shows the equality $\op{deg}W_{d,m}(x)=(d-1)m$.

\noindent (2) Because $\cal{H}_{d}(0)=1$ then $\cal{H}_{d}^{(m)}(0)=W_{d,m}(0)$ and thus $W_{d,m}(0)$ is just the coefficient near $x^m/m!$ in the power series expansion of $\cal{H}_{d}(x)$ in the neighborhood of the point $x=0$. Hence the result.

\noindent (3) The stated identity is a simple consequence of the definition of the sequence $(W_{d,m}(x))_{m\in\N}$. Indeed, in the case of odd $d$ we have $\cal{H}_{d}(x)^{-1}=\cal{H}_{d}(-x)$. Let us also note that
$$
\cal{H}_{d}^{(m)}(x)=\sum_{n=0}^{\infty}\frac{H_{d}(n+m)}{n!}x^{n}.
$$
In consequence, from the identity
$$
\cal{H}_{d}^{(m)}(x)\cal{H}_{d}(x)^{-1}=\cal{H}_{d}^{(m)}(x)\cal{H}_{d}(-x)=W_{d,m}(x)
$$
we see that the coefficient near $x^{n}$ in the power series expansion of the function $\cal{H}_{d}^{(m)}(x)\cal{H}_{d}(-x)$ is just the coefficient of $W_{d,m}(x)$ near $x^{n}$. But in this case the $n$-th coefficient is exactly the sum stated in the statement of our result.

\noindent (4) It is clear that the function $\cal{H}_{d}(x)W_{d,n}(x)$ is the coefficient near $t^{n}/n!$ in the power series expansion of the function $\cal{H}_{d}(x)\cal{W}_{d}(x,t)$ in the neighborhood of the point $t=x$. Indeed,
\begin{align*}
\cal{H}_{d}(x)\cal{W}_{d}(x,t)&=\cal{H}_{d}(x)\sum_{n=0}^{\infty}\frac{W_{d,n}(x)}{n!}t^{n}=\sum_{n=0}^{\infty}\frac{\cal{H}_{d}^{(n)}(x)}{n!}t^{n}\\
                &=\sum_{n=0}^{\infty}\frac{\cal{H}_{d}^{(n)}(x)}{n!}((t+x)-x)^{n}=\cal{H}_{d}(t+x).
\end{align*}
Multiplying both sides by $\cal{H}_{d}(x)^{-1}$ we get
$$
\cal{W}_{d}(x,t)=e^{-x-\frac{x^{d}}{d}}\cal{H}_{d}(x+t)=e^{t+\frac{(t+x)^{d}-x^{d}}{d}}
$$
and hence the result.

\end{proof}
As an immediate consequence of the shape of exponential generating function for $(W_{d,n}(x))_{n\in\N}$ we get the following:

\begin{cor}\label{recforW}
The sequence $(W_{d,n}(x))_{n\in\N}$ satisfies the following recurrence relation: $W_{d,n}(x)=0$ for $n<0$, $W_{d,0}(x)=1$ and
$$
W_{d,n}(x)=(1+x^{d-1})W_{d,n-1}(x)+\sum_{j=2}^{d}\binom{d-1}{j-1}(n-1)_{(j-1)}x^{d-j}W_{d,n-j}(x)
$$
for $n\geq 1$.
\end{cor}
\begin{proof}
By comparing the coefficients near $t^{n}$ in the identity
$$
\frac{\partial \cal{W}_{d}(x,t)}{\partial t}=(1+(t+x)^{d-1})\cal{W}_{d}(x,t)
$$
we get the equality
$$
\frac{1}{n!}W_{d,n+1}(x)=\frac{1}{n!}(1+x^{d-1})W_{d,n}(x)+\sum_{i=1}^{d-1}\frac{1}{(n-i)!}\binom{d-1}{i}W_{d,n-i}(x).
$$
Multiplying both sides by $n!$, writing $n!/(n-i)!$ as $(n)_{(i-1)}$ and replacing $n$ by $n-1$ and $i$ by $j-1$ we get the result (with the convention $W_{d,n}(x)\equiv 0$ for $n<0$).
\end{proof}

\begin{rem}
{\rm In the proof of Corollary \ref{recforW} we proved recurrence relation for the sequence $(W_{d,n}(x))_{n\in\N}$ by manipulating of exponential generating functions. One can also consider the (formal) ordinary generating function of the sequence $(W_{d,n}(x))_{n\in\N}$ and get a relation of a different kind. Indeed, it is an easy exercise to prove that the (formal) power series $\cal{S}_{d}(x,t)=\sum_{n=0}^{\infty}W_{d,n}(x)t^n$ satisfies the differential equation
$$
(1-(1+x^{d-1})t)\cal{S}_{d}(x,t)=1+t\frac{\partial \cal{S}_{d}(x,t)}{\partial x}.
$$
Multiplying both sides of the above identity by $(1-(1+x^{d-1})t)^{-1}$ and comparing coefficients near $t^n$ we easily deduce the following relation
$$
W_{d,n}(x)=(1+x^{d-1})^{n}+\sum_{k=1}^{n}(1+x^{d-1})^{n-k}W_{d,k-1}'(x).
$$
This gives us an expression of $W_{d,n}(x)$ in terms of the sequence $(W_{d,n}'(x))_{n\in\N}$.
}
\end{rem}

\begin{rem}
{\rm Let us note that the sequence $(W_{2,m}(x))_{m\in\N}$ was implicitly studied by Amdeberhan and Moll (Corollary 2.6 in \cite{AmdMoll}). Although the mentioned authors did not find recurrence relation for this sequence they were able to get the closed form
$$
W_{2,n}(x)=\sum_{i=0}^{n}\binom{n}{i}H_{2}(n-i)x^{i}.
$$
From our result we get $W_{2,0}(x)=1, W_{2,1}(x)=1+x$ and for $n\geq 2$ we have
$$
W_{2,n}(x)=(1+x)W_{2,n-1}(x)+(n-1)W_{2,n-2}(x).
$$
This allows us to deduce the following nice identity:
$$
W_{2,n}(-1)=\sum_{i=0}^{n}(-1)^{i}\binom{n}{i}H_{2}(n-i)=\begin{cases}
\begin{array}{lll}
0 &  & \mbox{if}\;n\equiv 1\pmod{2}\\
\frac{(2m)!}{2^{m}m!} &  & \mbox{if}\;n=2m
                         \end{array}
\end{cases}.
$$
Indeed, taking $x=-1$ in the recurrence relation for $W_{2,n}(x)$ we get $W_{2,n}(-1)=(n-1)W_{2,n-2}(-1)$. Because $W_{2,0}(-1)=1$ and $W_{2,1}(-1)=0$ by a simple iteration of the above recurrence we get the formula presented above.

Let us also note that the value of $W_{2,n}(x)$ at $x=1$ has striking combinatorial interpretation. Indeed, the sequence $(W_{2,n}(1))_{n\in\N}$ satisfies the recurrence relation: $W_{2,0}(1)=1, W_{2,1}(1)=2$ and
$$
W_{2,n}(1)=2W_{2,n-1}(1)+(n-1)W_{2,n-2}(1)\quad\mbox{for}\quad n\geq 2.
$$
In particular, $(W_{2,n}(1))_{n\in\N}$ is a binomial transform of the sequence $(H_{2}(n))_{n\in\N}$. Moreover, as was observed by Donaghey in \cite{Don}, this sequence enumerates the general switchboard problem, i.e., it enumerates the states of a telephone exchange with $n$ subscribers which is provided with means to connect subscribers singly to outside lines and in pairs internally (no conference circuits).
}
\end{rem}

The last results of this section are devoted to the behaviour of sequences $(W_{d,m}(x)\pmod{c})_{m\in\N}$ where $c=d-1$ or $c=d=p$ is a prime number.

\begin{thm}
The following congruence holds for each $d\geq 2$ and $m\in\N$:
$$W_{d,m}(x)\equiv (1+x^{d-1})^m\pmod{d-1}.$$
In particular, if $p$ is a prime number then
$$W_{p+1,m}(x)\equiv (1+x)^{pm}\pmod{p}$$
for any $m\in\N$.
\end{thm}

\begin{proof}
We proceed by induction on $m\in\N$. If $m=0$ then $W_{d,0}(x)=1$, hence the statement of our theorem is true. If $m>0$ then we use Corollary \ref{recforW} and the induction hypothesis for $m-1$ and get
\begin{align*}
W_{d,m}(x) &=(1+x^{d-1})W_{d,m-1}(x)+\sum_{j=2}^{d}\binom{d-1}{j-1}(m-1)_{(j-1)}x^{d-j}W_{d,m-j}(x)\\
&\equiv (1+x^{d-1})\cdot (1+x^{d-1})^{m-1}+\sum_{j=2}^{d}\binom{m-1}{j-1}(d-1)_{(j-1)}x^{d-j}W_{d,m-j}(x)\\
&\equiv (1+x^{d-1})^m\pmod{d-1}.
\end{align*}
\end{proof}

\begin{thm}\label{Wppmodp}
For each prime $p$ we have $W_{p,p}(x)\equiv x^{p(p-1)}\pmod{p}$.
\end{thm}

\begin{proof}
We easily check that we have $W_{2,2}(x)=x^2+2x+2$, so assertion holds for $p=2$. Now we assume that $p$ is an odd prime number. We will show that all the coefficients of the polynomial $W_{p,p}(x)$ except for the leading one are divisible by $p$. We use the formula
\begin{align*}
A_p(n,p) & =\frac{1}{n!}\sum_{k=0}^{n}\binom{n}{k}(-1)^{n-k}H_{p}(p+k)H_{p}(n-k)\\
& =\sum_{k=0}^{n}(-1)^{n-k}\frac{H_{p}(p+k)}{k!}\cdot\frac{H_{p}(n-k)}{(n-k)!},\quad 0\leq n\leq p(p-1)
\end{align*}
from the third part of the statement of Theorem \ref{Wdn}. It suffices to prove that for each $n\in\{0,...,p(p-1)-1\}$ and $k\in\{0,...,n\}$ we have $\nu_p\left(\frac{H_{p}(p+k)}{k!}\cdot\frac{H_{p}(n-k)}{(n-k)!}\right)\geq 1$. By Corollary \ref{vpHpnfrombelow}, $\nu_p(H_p(p+k))\geq 1+\left\lfloor\frac{k}{p}\right\rfloor$ and $\nu_p(H_p(n-k))\geq \left\lfloor\frac{n-k}{p}\right\rfloor$. On the other hand, it is easy to see that $\nu_p(k!)=\left\lfloor\frac{k}{p}\right\rfloor$ and $\nu_p((n-k)!)=\left\lfloor\frac{n-k}{p}\right\rfloor$. Summing up the above observations, we get $\nu_p\left(\frac{H_{p}(p+k)}{k!}\cdot\frac{H_{p}(n-k)}{(n-k)!}\right)\geq 1$ and the statement of our corollary follows for any odd prime $p$.
\end{proof}

\begin{thm}
The sequence
$$
(W_{p,m}(x)x^{p(1-p)\lfloor\frac{m}{p}\rfloor}\pmod{p})_{m\in\N}
$$
is periodic and the length of the period is $p$.

In particular
\begin{align*}
&(W_{2,m}(x)x^{-2\lfloor\frac{m}{2}\rfloor}\pmod{2})_{m\in\N}=\overline{(1,1+x)},\\
&(W_{3,m}(x)x^{-6\lfloor\frac{m}{3}\rfloor}\pmod{3})_{m\in\N}=\overline{(1,1+x^2,(x+2)(x^3+x^2))}.
\end{align*}
\end{thm}

\begin{proof}
First, we prove by induction on $m\in\N$ that
\begin{equation}\label{Wpmperiod}
W_{p,m+p}(x)\equiv x^{p(p-1)}W_{p,m}(x)\pmod{p},\quad m\in\N.
\end{equation}
The above congruence is satisfied for $m=0$ by the Theorem \ref{Wppmodp}. Assume now that (\ref{Wpmperiod}) holds for some $m\in\N$ and we will show that it holds for $m+1$. By (\ref{Wdmpol}) we get
\begin{align*}
W_{p,m+1+p}(x) & = (1+x^{p-1})W_{p,m+p}(x)+W'_{p,m+p}(x)\equiv (1+x^{p-1})x^{p(p-1)}W_{p,m}(x)+(x^{p(p-1)}W_{p,m}(x))'\\
& = (1+x^{p-1})x^{p(p-1)}W_{p,m}(x)+p(p-1)x^{p(p-1)-1}W_{p,m}(x)+x^{p(p-1)}W'_{p,m}(x)\\
& \equiv x^{p(p-1)}[(1+x^{p-1})W_{p,m}(x)+W'_{p,m}(x)] = x^{p(p-1)}W_{p,m+1}(x)\pmod{p}.
\end{align*}
Now we have congruence (\ref{Wpmperiod}) proved. Let us fix $m\in\N$ and write $m=jp+i$. Then, using (\ref{Wpmperiod}) $j$ times, we get
$$W_{p,m}(x)\equiv x^{p(p-1)j}W_{p,i}(x)\pmod{p}.$$
After multiplication by $x^{p(1-p)j}$ and substitution $j=\left\lfloor\frac{m}{p}\right\rfloor$ and $i=m\pmod p$ we obtain
$$W_{p,m}(x)x^{p(1-p)\left\lfloor\frac{m}{p}\right\rfloor}\equiv W_{p,m\pmod p}(x)\pmod{p}$$
and get the result
\end{proof}

\section{Divisors of numbers $H_d(n)$}\label{Section7}

In this section we study various properties of the numbers $H_d(n)$ and $H_d(n)-1$. We are interested in prime divisors and various $\text{GCD}$ properties.

\subsection{Prime divisors of numbers $H_d(n)$}

Let us denote $\mathbb{P}_d=\{p\in\mathbb{P}: p\mid H_{d}(n)\mbox{ for some }n\in\N\}$.

\begin{thm}
For each $d\in\N_{\geq 2}$ the set $\mathbb{P}_d$ is infinite.
\end{thm}

\begin{proof}
Assume by the contrary that $p_1<...<p_s$ are all the elements of the set $\bbb{P}_d$. We consider two cases. The first case is when $d$ is a composite number. We put $P=p_1\cdot ...\cdot p_s$. By Corollary \ref{Hdnmodc} we know that the sequence $(H_d(n)\pmod{P})_{n\in\N}$ is periodic of period $P$, as $2<d$ and thus $2$ is not a divisor of any number $H_d(n)$. In particular, the numbers $H_d(Pm)$, $m\in\N$, are congruent to $1$ modulo $P$. Hence these numbers have no prime divisors. This combined with the fact that $H_d(Pm)>0$ implies $H_d(Pm)=1$ for any $m\in\N$. However, this is a contradiction with the fact that the sequence $(H_d(n))_{n\in\N}$ is an ultimately strictly increasing sequence of positive integers.

The second case is when $d$ is a prime number. Then $p_1=d=H_d(d)$. We put $P=p_1^2\cdot p_2\cdot ...\cdot p_s$ and $Q=p_2\cdot ...\cdot p_s$. By Corollary \ref{Hdnmodc} we know that the sequence $(H_d(n)\pmod{Q})_{n\in\N}$ is periodic of the period $Q$, as $d$ is a prime number not dividing $Q$. In particular, since $Q\mid P$, we have $H_d(Pm)\equiv H_d(Pm+1)\equiv 1\pmod{p_i}$ for $m\in\N$ and $i\in\{2,...,s\}$. According to Corollary \ref{vpHpnfrombelow}, $\nu_{p_1}(H_d(Pm))=\nu_{p_1}(H_d(Pm+1))=Qm(p_1-1)$. Since $H_d(Pm)>0$ and $p_1,...,p_s$ are all possible divisors of the numbers $H_d(n)$, $n\in\N$, thus $H_d(Pm)=H_d(Pm+1)=p_1^{Qm(p_1-1)}$ for any $m\in\N$. However, $H_d(Pm+1)=H_d(Pm)+(Pm)_{(d-1)}H_d(Pm-d+1)>H_d(Pm)$ for $m\geq\frac{d-1}{P}$ - a contradiction.
\end{proof}

\begin{conj}\label{densityofdivisors}
For each $d\in\N_{\geq 2}$ the set $\mathbb{P}\bs\mathbb{P}_d$ is infinite and its asymptotic density in $\mathbb{P}$ is equal to $\frac{1}{e}$.
\end{conj}

The following heuristic reasoning allows us to claim the second statement in the conjecture above. If we fix a prime number $p$ and randomly choose a sequence $(a_n)_{n\in\N}$ such that the sequence $(a_n\pmod{p})_{n\in\N}$ has the period $p$, then the probability that $p$ does not divide any term of this sequence is equal to $\left(1-\frac{1}{p}\right)^p$. The probability tends to $\frac{1}{e}$ with $p\rightarrow +\infty$. Note that $p>d$ belongs to the set $\mathbb{P}\bs\mathbb{P}_d$ if and only if $p$ does not divide any number $H_d(n)$ for $n\in\{d,...,p-1\}$. Moreover, the sequence $(H_d(n)\pmod p)_{n\in\N}$ is periodic of the period $p$. Therefore we suppose that the probability that $p\in\mathbb{P}\bs\mathbb{P}_d$ equals to $\left(1-\frac{1}{p}\right)^{p-d}$. Thus, this probability tends to $\frac{1}{e}$ with $p\rightarrow +\infty$ and hence the asymptotic density of the set $\mathbb{P}\bs\mathbb{P}_d$ in the set $\mathbb{P}$ is expected to be equal to $\frac{1}{e}$.

\begin{rem}
{\rm We performed small numerical search in order to check whether our expectations are likely to be true. More precisely, for $d\in\{2,\ldots, 10\}, k\leq 4000$ and $n\in \{d,\ldots, p_{k}\}$ we computed the number of primes $p_i$ such that $p_{i}|H_{d}(n)$. In the table below we present the result of our computations.
$$\begin{array}{|c|ccccccccc|}
  \hline
  d                                      &  2   & 3    & 4    & 5    & 6    & 7    & 8    & 9    & 10 \\
  \hline
  N=|\mathbb{P}_{d}\cap\{p_1,...,p_{4000}\}|      & 2509 & 2523 & 2511 & 2485 & 2518 & 2469 & 2518 & 2518 & 2499 \\
  \hline
  N/2000                                 & 0.62725 & 0.63075& 0.62775& 0.62125& 0.6295& 0.61725& 0.6295& 0.6295 & 0.62475 \\
  \hline
\end{array}
$$
\begin{center}
Table 1. The number of integers $k\leq 4000$ such that $H_{d}(n)\equiv 0\pmod{p_{k}}$ for some positive integer $n\leq p_{k}$ and $d\in\{2,\ldots, 10\}$.
\end{center}

The numbers in the last row in the table above are quite close to the number $1 - 1/e\simeq 0.63212$, which is the conjectured density of the set $\mathbb{P}_{d}$ in the set $\mathbb{P}$.
}
\end{rem}

\subsection{Greatest common divisors of numbers $H_d(n)-1$ and $n$}

Performing numerical computations, we observed a quite interesting formula for the greatest common divisor of the number $H_d(n)$ and its index $n$.

\begin{thm}
Let $d,n\in\N$ with $d\geq 2$. Then the number $\text{GCD}(H_d(n)-1,n)$ is equal to
\begin{itemize}
\item $n$ if $d$ is a composite number not equal to $4$ or $d=4$ and $\nu_2(n)\neq 2$;
\item $\frac{n}{2}$ if $d=4$ and $\nu_2(n)=2$;
\item $\frac{n}{d^{\nu_d(n)}}$ if $d$ is a prime number.
\end{itemize}
\end{thm}

\begin{proof}
At first, let us prove the theorem for $d$ composite. Let us note the equality:
\begin{equation}\label{Hd(n)-1}
\begin{aligned}
H_d(n)-1 & =\sum_{k=1}^{\left\lfloor\frac{n}{d}\right\rfloor}\frac{(n)_{(dk)}}{k!d^k}=n\sum_{k=1}^{\left\lfloor\frac{n}{d}\right\rfloor}\frac{(n-1)_{(dk-1)}}{k!d^k}=n\sum_{k=1}^{\left\lfloor\frac{n}{d}\right\rfloor}\frac{(n-1)_{(dk-1)}}{(dk-1)!}\cdot\frac{(dk-1)!}{k!d^k}\\
& =n\sum_{k=1}^{\left\lfloor\frac{n}{d}\right\rfloor}{n-1\choose dk-1}\cdot\frac{(dk-1)_{((d-1)k-1)}}{d^k}.
\end{aligned}
\end{equation}
For $d\neq 4$ it suffices to show that $\frac{(dk-1)_{((d-1)k-1)}}{d^k}$ is an integer for each $k\in\left\{1,...,\left\lfloor\frac{n}{d}\right\rfloor\right\}$. If $k=1$, then
\begin{align*}
\frac{(dk-1)_{((d-1)k-1)}}{d^k}=\frac{(d-1)_{(d-2)}}{d}=\frac{(d-1)!}{d}\in\Z,
\end{align*}
as $d$ is a composite number greater than $4$. If $k=2$ then
\begin{align*}
\frac{(dk-1)_{((d-1)k-1)}}{d^k}=\frac{(2d-1)_{(2d-3)}}{d^2}=\frac{(2d-1)!}{2d^2}=\frac{(d-1)!}{d}\cdot\frac{d}{d}\cdot\frac{(d+1)\cdot ...\cdot (2d-1)}{2},
\end{align*}
where $\frac{(d-1)!}{d}$ is an integer and $2$ divides some of the $d-1$ factors $d+1,...,2d-1$ as $d\geq 6$. Assume that $k\geq 3$. In order to prove that $\frac{(dk-1)_{((d-1)k-1)}}{d^k}$ is an integer, we will show that $\nu_p\left(\frac{(dk-1)_{((d-1)k-1)}}{d^k}\right)\geq 0$ for each prime number $p$. Let us write:
\begin{align*}
\nu_p\left(\frac{(dk-1)_{((d-1)k-1)}}{d^k}\right)=\nu_p((dk-1)\cdot ...\cdot (k+1))-k\nu_p(d).
\end{align*}
We see that if $\nu_p(d)=0$ then $\nu_p\left(\frac{(dk-1)_{((d-1)k-1)}}{d^k}\right)\geq 0$. Thus, assume now that $p\mid d$. Then, we estimate $\nu_p((dk-1)\cdot ...\cdot (k+1))$ from below by the number of multiplicities of $p$ among the integers from $k+1$ to $dk-1$. This number is at least equal to $\left\lfloor\frac{(d-1)k-1}{p}\right\rfloor$. Hence,
\begin{equation}\label{vpestim}
\begin{split}
& \nu_p((dk-1)\cdot ...\cdot (k+1))-k\nu_p(d)\geq\left\lfloor\frac{(d-1)k-1}{p}\right\rfloor-k\nu_p(d)=\left\lfloor\frac{dk-k-1}{p}\right\rfloor-k\nu_p(d)\\
& =\frac{d}{p}k-\left\lceil\frac{k+1}{p}\right\rceil-k\nu_p(d)=k\left(\frac{d}{p}-\nu_p(d)\right)-\left\lceil\frac{k+1}{p}\right\rceil\geq k\left(\frac{d}{p}-\nu_p(d)\right)-\frac{k+1}{p}-1.
\end{split}
\end{equation}
If $\nu_p(d)\geq 3$ or $p>2$ and $\nu_p(d)\geq 2$ then $d\geq p^{\nu_p(d)}$. Thus the last expression in (\ref{vpestim}) is bounded from below by $k(p^{\nu_p(d)-1}-\nu_p(d))-\frac{k+p+1}{p}$. 
We have the following chain of equivalences:
\begin{align*}
& k(p^{\nu_p(d)-1}-\nu_p(d))-\frac{k+p+1}{p}\geq 0\\
\Leftrightarrow & k(p^{\nu_p(d)-1}-\nu_p(d))\geq \frac{k+p+1}{p}\\
\Leftrightarrow & p^{\nu_p(d)-1}-\nu_p(d)\geq \frac{k+p+1}{kp}.
\end{align*}
The left hand side of the last inequality is at least equal to $1$. Since $k\geq 3$ and $p\in\bbb{P}$ then $\frac{1}{p}+\frac{1}{k}+\frac{1}{kp}\leq 1$. In consequence $\nu_p\left(\frac{(dk-1)_{((d-1)k-1)}}{d^k}\right)\geq 0$ for any prime $p$ such that $\nu_p(d)\geq 3$ or $p>2$ and $\nu_p(d)\geq 2$.

If $\nu_p(d)=1$ or $p=2$ and $\nu_2(d)=2$ then $d\geq 2p^{\nu_p(d)}$. Thus the last expression in (\ref{vpestim}) is bounded from below by $k(2p^{\nu_p(d)-1}-\nu_p(d))-\frac{k+p+1}{p}$. 
We have the following chain of equivalences:
\begin{align*}
& k(2p^{\nu_p(d)-1}-\nu_p(d))-\frac{k+p+1}{p}\geq 0\\
\Leftrightarrow & k(2p^{\nu_p(d)-1}-\nu_p(d))\geq \frac{k+p+1}{p}\\
\Leftrightarrow & 2p^{\nu_p(d)-1}-\nu_p(d)\geq \frac{k+p+1}{kp}.
\end{align*}
The left hand side of the last inequality is at least equal to $1$ and the right hand side, as we have seen earlier, is at most equal to $1$. Validity of the last inequality implies that $\nu_p\left(\frac{(dk-1)_{((d-1)k-1)}}{d^k}\right)\geq 0$ for any prime $p$ such that $\nu_p(d)=1$ or $p=2$ and $\nu_2(d)=2$.

We have just proven that $\frac{(dk-1)_{((d-1)k-1)}}{d^k}$ is an integer for each $k\in\left\{1,...,\left\lfloor\frac{n}{d}\right\rfloor\right\}$. Thus $\sum_{k=1}^{\left\lfloor\frac{n}{d}\right\rfloor}{n-1\choose dk-1}\cdot\frac{(dk-1)_{((d-1)k-1)}}{d^k}\in\Z$ and, as a result, $n\mid H_d(n)-1$ in the case of $d>4$ composite.

\bigskip

Let us consider now the case of $d=4$. By (\ref{Hd(n)-1}) we know that $H_4(n)-1=n\sum_{k=1}^{\left\lfloor\frac{n}{4}\right\rfloor}{n-1\choose 4k-1}\cdot\frac{(4k-1)!}{k!4^k}$. If $p$ is an odd prime number, then $\nu_p\left(\frac{(4k-1)!}{k!4^k}\right)=\nu_p((4k-1)_{(3k-1)})\geq 0$. We have
\begin{align*}
& \nu_2\left(\frac{(4k-1)!}{k!4^k}\right)=4k-1-s_2(4k-1)-k+s_2(k)-2k=k-1-s_2(4(k-1)+3)+s_2(k)\\
& =k-1-s_2(k-1)-2+s_2(k)=k-2-s_2(k-1)-s_2(1)+s_2(k)=k-2-\nu_2\left(k\choose 1\right)\\
& =k-2-\nu_2(k).
\end{align*}
In the above computation we used the Legendre formula $\nu_2(m!)=m-s_2(m)$ and its corollary $\nu_2\left(m\choose l\right)=s_2(l)+s_2(m-l)-s_2(m)$ for $0\leq l\leq m$. The value $k-2-\nu_2(k)$ is non-negative for $k\geq 3$ and equal to $-1$ for $k\in\{1,2\}$. Hence, the value of ${n-1\choose 4k-1}\cdot\frac{(4k-1)!}{k!4^k}$ is an integer for $k\geq 3$. Meanwhile for $k\in\{1,2\}$ we compute $\nu_2\left(n-1\choose 4k-1\right)$. For $k=1$ we have
\begin{align*}
\nu_2\left(n-1\choose 4k-1\right)=\nu_2\left(n-1\choose 3\right)=\nu_2\left(\frac{(n-1)(n-2)(n-3)}{6}\right)=\nu_2((n-1)(n-2)(n-3))-1
\end{align*}
which is equal to $0$ if and only if $4\mid n$. For $k=2$ we have
\begin{align*}
\nu_2\left(n-1\choose 4k-1\right)=\nu_2\left(n-1\choose 7\right)=\nu_2\left(\frac{(n-1)\cdot ...\cdot(n-7)}{7!}\right)
\end{align*}
which is equal to $0$ if and only if $8\mid n$. If $\nu_2(n)<2$ then $\nu_2\left(n-1\choose 4k-1\right)\geq 1$ for $k\in\{1,2\}$ and, as a consequence, $\nu_2\left({n-1\choose 4k-1}\cdot\frac{(4k-1)!}{k!4^k}\right)\geq 0$ for any $k\leq\left\lfloor\frac{n}{4}\right\rfloor$. If $\nu_2(n)=2$ then $\nu_2\left({n-1\choose 4k-1}\cdot\frac{(4k-1)!}{k!4^k}\right)\geq 0$ for $k>1$ and $\nu_2\left({n-1\choose 4k-1}\cdot\frac{(4k-1)!}{k!4^k}\right)=-1$ for $k=1$. Thus $\nu_2\left(\sum_{k=1}^{\left\lfloor\frac{n}{4}\right\rfloor}{n-1\choose 4k-1}\cdot\frac{(4k-1)!}{k!4^k}\right)=-1$. If $\nu_2(n)>2$ then $\nu_2\left({n-1\choose 4k-1}\cdot\frac{(4k-1)!}{k!4^k}\right)\geq 0$ for $k>2$ and $\nu_2\left({n-1\choose 4k-1}\cdot\frac{(4k-1)!}{k!4^k}\right)=-1$ for $k\leq 2$. Thus $\nu_2\left(\sum_{k=1}^{\left\lfloor\frac{n}{4}\right\rfloor}{n-1\choose 4k-1}\cdot\frac{(4k-1)!}{k!4^k}\right)\geq 0$. Summing up, if $\nu_2(n)\neq 2$ then $\left(\sum_{k=1}^{\left\lfloor\frac{n}{4}\right\rfloor}{n-1\choose 4k-1}\cdot\frac{(4k-1)!}{k!4^k}\right)$ is an integer and then $n\mid H_4(n)-1$. If $\nu_2(n)=2$ then $2\left(\sum_{k=1}^{\left\lfloor\frac{n}{4}\right\rfloor}{n-1\choose 4k-1}\cdot\frac{(4k-1)!}{k!4^k}\right)$ is an integer and then $\frac{n}{2}\mid H_4(n)-1$ while $n\nmid H_4(n)-1$. The proof for the case of $d=4$ is finished.

\bigskip

Let us consider the case of prime $d$. If $n<d$ then $H_d(n)-1=0$ and obviously $n\mid H_d(n)-1$. If $n\geq d$ then $d\mid H_d(n)$ and thus $d\nmid H_d(n)-1$. Let us write
\begin{align*}
 & H_d(n)-1=\sum_{k=1}^{\left\lfloor\frac{n}{d}\right\rfloor}\frac{(n)_{(dk)}}{k!d^k}=\frac{n}{d^{\nu_d(n)}}\cdot d^{\nu_d(n)}\sum_{k=1}^{\left\lfloor\frac{n}{d}\right\rfloor}\frac{(n-1)_{(dk-1)}}{k!d^k}\\
& =\frac{n}{d^{\nu_d(n)}}\cdot d^{\nu_d(n)}\sum_{k=1}^{\left\lfloor\frac{n}{d}\right\rfloor}\frac{(n-1)_{(dk-1)}}{(dk-1)!}\cdot\frac{(dk-1)!}{k!d^k}=\frac{n}{d^{\nu_d(n)}}\cdot d^{\nu_d(n)}\sum_{k=1}^{\left\lfloor\frac{n}{d}\right\rfloor}{n-1\choose dk-1}\cdot\frac{(dk-1)_{((d-1)k-1)}}{d^k}.
\end{align*}
Since $d\nmid \frac{n}{d^{\nu_d(n)}}$, we thus have $\nu_d\left(d^{\nu_d(n)}\sum_{k=1}^{\left\lfloor\frac{n}{d}\right\rfloor}{n-1\choose dk-1}\cdot\frac{(dk-1)_{((d-1)k-1)}}{d^k}\right)=\nu_d(H_d(n)-1)=0.$
If $p$ is a prime number not equal to $d$ then
\begin{align*}
\nu_p\left(\frac{(dk-1)_{((d-1)k-1)}}{d^k}\right)=\nu_p((dk-1)_{((d-1)k-1)})\geq 0
\end{align*}
and $\nu_p\left(d^{\nu_d(n)}\sum_{k=1}^{\left\lfloor\frac{n}{d}\right\rfloor}{n-1\choose dk-1}\cdot\frac{(dk-1)_{((d-1)k-1)}}{d^k}\right)\geq 0$. Finally, $$d^{\nu_d(n)}\sum_{k=1}^{\left\lfloor\frac{n}{d}\right\rfloor}{n-1\choose dk-1}\cdot\frac{(dk-1)_{((d-1)k-1)}}{d^k}\in\Z,$$ which means that $\frac{n}{d^{\nu_d(n)}}\mid H_d(n)-1$ and $\frac{n}{d^{\nu_d(n)}}=\text{GCD}(H_d(n)-1,n)$.
\end{proof}

\subsection{Nontrivial common divisors of numbers $H_a(n)$ and $H_b(n)$}

Let us fix two distinct integers $a,b\geq 2$. It is natural to ask about existence of indices $n$ for which the numbers $H_a(n)$ and $H_b(n)$ have a common divisor greater than $1$. The above question can be answered immediately. Let $d\in\N_{\geq 2}$, $n\in\N_{\geq d}$ and $p$ be a prime divisor of $H_d(n)$. Then $p\mid\text{GCD}(H_d(n+p),H_p(n+p))$. Indeed, it is a simple consequence of periodicity of the sequence $(H_d(n)\pmod{p})_{n\in\N}$ and divisibility properties of the sequence $(H_p(n))_{n\in\N}$ for prime number $p$. Moreover, we give the following.

\begin{thm}\label{H_d(3d+2)orH_d(4d)}
Let $d\in\N_{\geq 3}$. Then $2d+1\mid H_d(3d+2)$ if and only if $2d+1\mid d!-1$.

If additionally $d$ is odd then
\begin{itemize}
\item $2d+1\mid H_{2d+1}(3d+2)$ if and only if $2d+1\mid (d!)^2-1$.
\item $2d+1\mid H_d(4d)$ if and only if $2d+1\mid d!+1$;
\item $2d+1\mid H_{2d+1}(4d)$ if and only if $2d+1\mid (d!)^2-1$.
\end{itemize}
\end{thm}

\begin{proof}
Since the proofs of the remaining equivalences are very similar, we only show the second equivalence.


By the exact formula for $H_{2d+1}(3d+2)$ we have $H_{2d+1}(3d+2)=1+\frac{(3d+2)_{(2d+1)}}{2d+1}$. We compute
\begin{align*}
& 1+(3d+2)\cdot ...\cdot (2d+2)(2d)\cdot ...\cdot (d+2)\equiv 1+(-d)d\cdot ...\cdot 1\cdot(-1)\cdot ...\cdot (1-d)\\
& \equiv 1+ (-1)^{d}(d!)^2\equiv 1-(d!)^2\pmod{2d+1}
\end{align*}
and we see that the divisibility of $H_{2d+1}(3d+2)$ by $2d+1$ is equivalent to the divisibility of $(d!)^2-1$ by $2d+1$.


\end{proof}

In the above theorem we needed assumption that $d$ is odd. The next proposition shows that primality of $2d+1$ is necessary and sufiicient for validity of the condition $2d+1\mid (d!)^2-1$.

\begin{prop}
Let $d$ be an odd positive integer. Then $2d+1\mid (d!)^2-1$ if and only if $2d+1$ is a prime number.
\end{prop}

\begin{proof}
If $2d+1$ is prime then $2d+1\mid H_{2d+1}(3d+2)$ as $3d+2\geq 2d+1$. Thus, by Theorem \ref{H_d(3d+2)orH_d(4d)} there holds $2d+1\mid (d!)^2-1$.

If $2d+1$ is a composite number then by Corollary \ref{Hdnmodc} we have $H_{2d+1}(3d+2)\equiv H_{2d+1}(d+1)=1\pmod{2d+1}$ as $d+1<2d+1$. Hence by Theorem \ref{H_d(3d+2)orH_d(4d)} the condition $2d+1\mid (d!)^2-1$ is not satisfied.
\end{proof}

Now we are ready to give a formula for the one-parametr infinite family of quadruples $(a,b,n,c)$ such that $c>1$ divides both $H_a(n)$ and $H_b(n)$.

\begin{cor}
If $n\geq 3$ is an odd integer such that $2n+1$ is a prime number then $2n+1$ is a common divisor of numbers $H_n(3n+2)$ and $H_{2n+1}(3n+2)$ or $H_n(4n)$ and $H_{2n+1}(4n)$. More precisely,
\begin{itemize}
\item $2n+1\mid\text{GCD}(H_n(3n+2), H_{2n+1}(3n+2))$ if and only if $2n+1\mid n!-1$;
\item $2n+1\mid\text{GCD}(H_n(4n), H_{2n+1}(4n))$ if and only if $2n+1\mid n!+1$.
\end{itemize}
In other words, if $(a,b,n,c)=(d,2d+1, [2d+1\mid d!-1]\cdot (3d+2)+[2d+1\mid d!+1]\cdot (4d), 2d+1)$, where $d$ is an odd positive integer such that $2d+1$ is a prime number and we use the Iverson bracket $$[\varphi]=\begin{cases}1, & \mbox{ if } \varphi \mbox{ is true}\\0, & \mbox{ if } \varphi \mbox{ is false}\end{cases},$$ then $c\mid\text{GCD}(H_a(n),H_b(n))$.
\end{cor}

Let us observe that all the examples of quadruples $(a,b,n,c)$ with the property that $c>1$ divides both $H_a(n)$ and $H_b(n)$ are such that $c=b$ is a prime number and $n\geq b$. This is a motivation to formulate the following.

\begin{ques}
Are there infinitely many triples $(a,b,c)\in\N^3$, $a,b,c\geq 2$, such that $c>1$ is coprime to $ab$ and $c\mid\text{GCD}(H_a(n),H_b(n))$ for some $n\in\N$?
\end{ques}

Moreover, it is quite natural to ask the following.

\begin{ques}
Are there infinitely many pairs $(a,b)\in\N^2$, $a,b\geq 2$, such that $\text{GCD}(H_a(n),H_b(n))=1$ for each $n\in\N$?
\end{ques}

\section{Some polynomials related to numbers $H_d(n)$ and their combinatorial interpretation}\label{Section8}

Let us define polynomials $H_{d}(n,x)\in\Z[x]$ recursively: $H_{d}(n,x)=x^n$ for $n\in\{0,...,d-1\}$ and
$$H_{d}(n,x)=xH_{d}(n-1,x)+(n-1)\cdot...\cdot(n-d+1)H_{d}(n-d,x)\mbox{ for } n\geq d.$$
Then $H_{d}(n,1)=H_{d}(n)$ and we have exact formula
$$H_{d}(n,x)=\sum_{j=0}^{\left\lfloor\frac{n}{d}\right\rfloor} \frac{n!}{(n-dj)!j!d^j}x^{n-dj}.$$

\begin{prop}
The exponential generating function of the sequence $(H_{d}(n,x))_{n\in\N}$ takes the form
$$\cal{H}_d(x,t)=\sum_{n=0}^{+\infty}\frac{H_d(n,x)}{n!}t^n=e^{xt+\frac{t^d}{d}}.$$
\end{prop}

\begin{proof}
Let us write $e^{xt+\frac{t^d}{d}}=\sum_{n=0}^{+\infty}\frac{a_d(n,x)}{n!}t^n$ By comparing the coefficients near $t^{n-1}$ in the identity
$$\frac{\partial e^{xt+\frac{t^d}{d}}}{\partial t}=(x+t^{d-1})e^{xt+\frac{t^d}{d}}$$
and multiplying by $(n-1)!$ we get the recurrence
$$a_{d}(n,x)=xa_{d}(n-1,x)+(n-1)\cdot...\cdot(n-d+1)a_{d}(n-d,x)\mbox{ for } n\geq d.$$
The above recurrence combined with the fact that
$$a_d(n,x)=\left.\frac{\partial^n e^{xt+\frac{t^d}{d}}}{\partial t^n}\right|_{t=0}=x^n$$
for $n\in\{0,...,d-1\}$ implies $a_d(n,x)=H_d(n,x)$ for each $n\in\N$.
\end{proof}

\begin{thm}
Let $\text{Fix}_{n}(\sigma)$ and $S_{d,n}$ denote the set of fixed points of the permutation $\sigma\in S_n$ and the set of permutations being product of pairwise disjoint $d$-cycles, respectively. Then, we have the following equality
$$H_{d}(n,x)=\sum_{\sigma\in S_{d,n}} x^{\#\text{Fix}_n(\sigma)}.$$
\end{thm}

\begin{proof}
Let us write the right hand side of the above equality as $g_{d,n}(x)$. We will show that the sequence $(g_{d,n}(x))_{n\in\N}$ satisfies the same recurrence as the sequence $(H_d(n,x))_{n\in\N}$. This fact combined with obvious equality $g_{d,n}(x)=x^n$ for $n\in\{0,...,d-1\}$ will give us the statement of our theorem.

Let us fix $\sigma\in S_{d,n}$. If $\sigma(n)=n$ then $\sigma=\sigma_1$ for some $\sigma_1\in S_{d,n-1}$. The number of fixed points of $\sigma$ is equal to the number of fixed points of $\sigma_1$ increased by $1$, since $n$ is an extra fixed point of $\sigma$. Hence the summand $x^{\#\text{Fix}_n(\sigma)}$ in $g_{d,n}(x)$ reduces with the summand $x^{\#\text{Fix}_n(\sigma_1)+1}$ in $xg_{d,n-1}(x)$. If $\sigma(n)\neq n$ then $\sigma=\sigma_1\circ\pi$, where $\pi$ is a $d$-cycle containing $n$ and $\sigma_1$ is a product of pairwise disjoint $d$-cycles on a set $\{1,...,n\}\bs\mbox{supp}(\pi)$. Then $\sigma_1$ can be treated as a member of the set $S_{d,n-d}$ and $\#\text{Fix}_{n-d}(\sigma_1)=\#\text{Fix}_n(\sigma)$. Moreover, $\pi$ can be chosen in $(n-1)_{(d-1)}$ ways. Thus the summand $x^{\#\text{Fix}_n(\sigma)}$ in $g_{d,n}(x)$ reduces with the summand $x^{\#\text{Fix}_n(\sigma_1)}$ in $(n-1)_{(d-1)}g_{d,n-d}(x)$. Summing over all $\sigma\in S_{d,n}$ we obtain $g_{d,n}(x)=xg_{d,n-1}(x)+(n-1)_{(d-1)}g_{d,n-d}(x)$, which completes the proof.
\end{proof}

For given $d\in\N_{\geq 2}$, the sequence of polynomials $(H_{d}(n,x))_{n\in\N}$ contains more precise information concerning the number of permutations in $S_{n}$ which are product of disjoint $d$-cycles. However, because our primary object of study is the sequence $(H_{d}(n))_{n\in\N}$,
we offer only few results concerning this natural polynomial generalization. First we prove the following:

\begin{thm}\label{Hd(n,x)mod(d)}
Let $d\in\N_{\geq 2}$. Then the following congruence holds for any $n\in\N$:
\begin{itemize}
\item $H_d(n,x)\equiv x^n\pmod{d}$ if $d$ is a composite number greater than $4$;
\item $H_d(n,x)\equiv x^n+2\left\lfloor\frac{n}{4}\right\rfloor x^{n-4}=x^{n-4\cdot\left(\left\lfloor\frac{n}{4}\right\rfloor\pmod{2}\right)}(x^4+2)^{\left\lfloor\frac{n}{4}\right\rfloor\pmod{2}}\pmod{d}$ if $d=4$;
\item $H_d(n,x)\equiv x^{n\pmod{d}}(x-1)^{d\left\lfloor\frac{n}{d}\right\rfloor}\pmod{d}$ if $d$ is a prime number.
\end{itemize}
\end{thm}

\begin{proof}
We will proceed by induction on $n\in\N$. The statement of the theorem is true for $n\in\{0,...,d-1\}$. Let us suppose that $d>4$ is a composite number and $n\geq d$. Then the product of consecutive $d-1$ integers is divisible by $d$, thus by the recurrence formula and inductive hypothesis for $n-1$ we have
\begin{align*}
H_d(n,x)=xH_d(n-1,x)+(n-1)_{(d-1)}H_d(n-d,x)\equiv x\cdot x^{n-1}=x^n\pmod{d}.
\end{align*}

Suppose now that $d=4$ and $n\geq 4$. If $4\mid n$ then, by the recurrence defining polynomials $H_4(n,x)$, inductive hypothesis for $n-1$ and $n-4$ and the fact that $(n-1)_{(3)}\equiv 3!\equiv 2\pmod{4}$, we obtain the following:
\begin{align*}
& H_4(n,x) = xH_4(n-1,x)+(n-1)_{(3)}H_4(n-4,x)\\
& \equiv x^n+2\left\lfloor\frac{n-1}{4}\right\rfloor x^{n-4}+2\cdot\left(x^{n-4}+2\left\lfloor\frac{n-4}{4}\right\rfloor x^{n-8}\right)\equiv x^n+2\left\lfloor\frac{n-1}{4}\right\rfloor x^{n-4}+2\cdot x^{n-4}\\
& =x^n+2\left(\left\lfloor\frac{n-1}{4}\right\rfloor +1\right)x^{n-4}=x^n+2\left\lfloor\frac{n}{4}\right\rfloor x^{n-4}\pmod{4}.
\end{align*}

If $4\nmid n$ then, by the recurrence defining polynomials $H_4(n,x)$, inductive hypothesis for $n-1$ and the fact that among the numbers $n-3$, $n-2$, $n-1$ there is a one divisible by $4$, we obtain the following:
\begin{align*}
& H_4(n,x) = xH_4(n-1,x)+(n-1)_{(3)}H_4(n-4,x)\equiv x^n+2\left\lfloor\frac{n-1}{4}\right\rfloor x^{n-4}\\
& =x^n+2\left\lfloor\frac{n}{4}\right\rfloor x^{n-4}\pmod{4}.
\end{align*}
Since the polynomial $x^4+2$ is irreducible over the ring $\Z/4\Z$, thus $H_4(n,x)\equiv x^n+2\left\lfloor\frac{n}{4}\right\rfloor x^{n-4}=x^{n-4\cdot\left(\left\lfloor\frac{n}{4}\right\rfloor\pmod{2}\right)}(x^4+2)^{\left\lfloor\frac{n}{4}\right\rfloor\pmod{2}}\pmod{4}$ is a factorization of the polynomial $H_4(n,x)$ modulo $4$.

Assume now that $d$ is a prime number and $n\geq d$ and the statement is true for any non-negative integer less than $n$. If $d\mid n$ then
\begin{align*}
H_d(n,x) & = xH_d(n-1,x)+(n-1)\cdot ...\cdot (n-d+1)H_d(n-d,x)\\
&\equiv x^{1+(n-1)\pmod{d}}(x-1)^{d\left\lfloor\frac{n-1}{d}\right\rfloor}+(d-1)\cdot ...\cdot 1\cdot x^{(n-d)\pmod{d}}(x-1)^{d\left\lfloor\frac{n-d}{d}\right\rfloor}\\
& = x^{1+(d-1)}(x-1)^{d\left(\frac{n}{d}-1\right)}+(d-1)!(x-1)^{d\left(\frac{n}{d}-1\right)}\\
& \equiv x^d(x-1)^{n-d}-(x-1)^{n-d}=(x^d-1)(x-1)^{n-d}\equiv (x-1)^d(x-1)^{n-d}\\
& =(x-1)^n=(x-1)^{d\cdot\left\lfloor\frac{n}{d}\right\rfloor}\pmod{d}.
\end{align*}

If $d\nmid n$ then, by the recurrence defining polynomials $H_d(n,x)$, inductive hypothesis for $n-1$ and the fact that among the numbers $(n-d+1)$, ..., $(n-1)$ there is a one divisible by $d$, we obtain the following:
\begin{align*}
& H_d(n,x) = xH_d(n-1,x)+(n-1)\cdot ...\cdot (n-d+1)H_d(n-d,x)\\
&\equiv x^{1+(n-1\pmod{d})}(x-1)^{d\left\lfloor\frac{n-1}{d}\right\rfloor}\equiv x^{n\pmod{d}}(x-1)^{d\left\lfloor\frac{n}{d}\right\rfloor}\pmod{d}.
\end{align*}
\end{proof}

Now, we concentrate on the case $d=2$ and prove the following

\begin{thm}\label{positivity}
Let $U(n,x)=H_{2}(n-1,x)H_{2}(n+1,x)-H_{2}(n,x)^2$ for $n\in\N_{+}$.
\begin{enumerate}
\item We have $U(n,x)=V(n,x^2)$, where $V(n,x)\in\Z[x]$ is of degree $n-1$.

\item We have $V(1,x)=1, V(2,x)=x-1, V(3,x)=x^2+3$ and  for $n\geq 4$ the following recurrence relation is true:
\begin{equation}\label{relationV}
V(n,x)=(x+n-3)V(n-1,x)+(n-2)(x+n)V(n-2,x)-(n-3)(n-2)^2V(n-3,x).
\end{equation}
\item
We have the following identity
$$
\cal{V}(x,t)=\sum_{n=0}^{\infty}\frac{V(n+1,x)}{n!}t^{n}=\frac{e^{\frac{xt}{1-t}}}{(1+t)\sqrt{1-t^2}}.
$$
\item
Let us write $V(n,x)=\sum_{i=0}^{n-1}a(i,n)x^{i}$. Then
$$
\begin{array}{lll}
a(0,2n)=V(2n,0)=-\left(\frac{(2n)!}{2^{n}n!}\right)^{2}, &  & a(0,2n+1)=V(2n+1,0)=\frac{1}{2n+1}\left(\frac{(2n+1)!}{2^{n}n!}\right)^{2},\\
a(1,2n)=V'(2n,0)=-a(0,2n), &  & a(1,2n+1)=V'(2n+1,0)=0.
\end{array}
$$
and
$$
a(i,n)>0 \quad\mbox{for each}\quad i\in\{2,\ldots,n-1\}.
$$
\end{enumerate}
\end{thm}
\begin{proof}
From the expression for $H_{2}(n,x)$ we see that it is an odd polynomial if $n\equiv 1\pmod{2}$ and an even polynomial in case of $n\equiv 0\pmod{2}$. We thus get that the polynomial $U(n,x)=H_{2}(n-1,x)H_{2}(n+1,x)-H_{2}(n,x)^2$ is an even polynomial and for each $n\in\N_{+}$ we can write $U(n,x)=V(n,x^2)$ for some polynomial $V(n,x)$. The expressions for $V(n,x)$ for $n\in\{1, 2, 3\}$ are clear from the shape of $H_{2}(n,x)$ for $n\in\{1,2,3,4\}$.

Let us observe that the sequence $(H_{2}(n,x))_{n\in\N}$ is a holonomic sequence of polynomials. Indeed, this is consequence of the fact that the generating function $\cal{H}_{d}(x,t)$ (as a function in variable $t$) is $D$-finite for each $d$, i. e. is a solution of some linear differential equation with polynomial coefficients. We recall that a sequence is holonomic if satisfies the recurrence relation with coefficients being polynomials in $n$. It is well known that if $(f_{n})_{n\in\N}, (g_{n})_{n\in\N}$ are holonomic sequences then the sequences $(f_{n}\pm g_{n})_{n\in\N}, (f_{n}g_{n})_{n\in\N}$ are holonomic too. It is straightforward (but a bit tedious) exercise to check that the sequence $(U(n,x))_{n\in\N_{+}}$ satisfies the relation given in the statement of our theorem (with $x$ replaced by $x^2$) and hence the sequence $(V(n,x))_{n\in\N_{+}}$ satisfies exactly the relation presented above.

If we define $\cal{V}(x,t)=\sum_{n=0}^{\infty}\frac{V(n+1,x)}{n!}t^{n}$, then it is an easy task to produce the differential equation satisfied by the function $\cal{V}(x,t)$. Indeed, applying standard methods and the relation (\ref{relationV}) we check that
$$
(1-t)^2(1+t)\frac{\partial \cal{V}(x,t)}{\partial t}+((1-3t+2t^2)-(1+t)x)\cal{V}(x,t)=0.
$$
The initial condition $\cal{V}(x,0)=1$ allow us to compute the unique solution of the above differential equation in the form
$$
\cal{V}(x,t)=\frac{e^{\frac{xt}{1-t}}}{(1+t)\sqrt{1-t^2}},
$$
which is exactly the expression presented in the statement of our theorem. Now, it is an easy task to compute the coefficient $a(i,n)$ for $i=0,1$. We have $a(0,1)=1, a(0,2)=-1, a(1,2)=1, a(0,3)=3, a(1,3)=0$ and our formulas are true in case of $n=1, 2, 3$. We see that the sequence $(a(0,n))_{n\in\N}$ satisfies the recurrence relation of degree 3 given by
$$
a(0,n)=(n-3)a(0,n-1)+(n-2)na(0,n-2)-(n-3)(n-2)^2a(0,n-3).
$$
The above relation comes from the relation (\ref{relationV}) by taking $x=0$. However, a closer look reveals that our sequence satisfies simpler recurrence relation
$$
a(0,n)=-a(0,n-1)+(n-1)(n-2)a(0,n-2),
$$
and a direct check shows that the expressions for $a(0,2n), a(0,2n+1)$ presented in the statement satisfy the above recurrence relation. Similarly, by differentiating the recurrence relation (\ref{relationV}) with respect to $x$ and taking the limit $x\rightarrow 0$ we get that the sequence $(a(1,n))_{n\in\N_{\geq 2}}$ satisfies the recurrence relation
$$
a(1,n)=a(0,n)+(n-1)a(1,n-1)+(n-2)a(0,n-2)+(n-2)na(1,n-2)-(n-3)^2(n-1)a(1,n-3).
$$
A direct check confirms that the expressions for $a(1,2n), a(1,2n+1)$ given in the statement are correct. We omit simple details.

Having proved the expressions for $a(0,n), a(1,n)$ we are ready to prove positivity of $a(i,n)$ for $i\in\{2,\ldots,n-1\}$. First, we get rid of the $0$-th coefficient in $V(n+1,x)$ by differentiating this polynomial with respect to $x$. Thus, by differentiating the generating function $\cal{V}(x,t)$ with respect to $x$, we arrive to the identity
\begin{equation}\label{identity}
\frac{\partial \cal{V}(x,t)}{\partial x}=\sum_{n=0}^{\infty}\frac{V'(n+1,x)}{n!}t^{n}=\frac{te^{\frac{xt}{1-t}}}{(1-t^2)\sqrt{1-t^2}}.
\end{equation}
Dividing the last equality by $t$, comparing expansions on both sides of the above identity, we see that the polynomial $V'(n+2,x)/(n+1)$ is just the binomial convolution of the integer sequence $(h_{n})_{n\in\N}$ and the sequence of polynomials $(R_{n}(x))_{n\in\N}$, where
$$
F_{1}(t^2)=\frac{1}{(1-t^2)\sqrt{1-t^2}}=\sum_{n=0}^{\infty}\frac{h_{n}}{n!}t^{2n},\quad\quad F_{2}(x,t)=e^{\frac{xt}{1-t}}=\sum_{n=0}^{\infty}\frac{R_{n}(x)}{n!}t^{n}.
$$
More precisely,
\begin{equation}\label{formula}
\frac{V'(n+2,x)}{n+1}=\sum_{i=0}^{\left\lfloor\frac{n}{2}\right\rfloor}\frac{(2i)!}{i!}\binom{n}{2i}h_{i}R_{n-2i}(x).
\end{equation}
The power series expansion of $F_{1}(t)$ is well known and we get that
$$
h_{n}=\frac{(2n+1)!}{2^{2n}n!}.
$$
In consequence, we see that the coefficients of the power series expansion of $F_{1}(t^2)$ are non-negative. Next, from the definition of $F_{2}(t)$ we see that $R_{0}(x)=1$. Moreover, we have that
$$
\frac{\partial F_{2}(x,t)}{\partial t}=\frac{x}{(1-t)^2}F_{2}(t).
$$
The above equality, together with the power series expansion $1/(1-t)^2=\sum_{n=0}^{\infty}(n+1)t^{n}$ allow us to get the recurrence relation satisfied by the sequence $(R_{n}(x))_{n\in\N}$. Indeed, by comparing the coefficients near $t^{n}$ we easily get
$$
R_{n}(x)=x\sum_{i=0}^{n-1}\binom{n}{i}(i+1)!R_{n-1-i}(x).
$$
Let us observe that $R_{n}(0)=0$ for $n\geq 1$ and $x=0$ is a simple root. Indeed, the simplicity of the root $x=0$ is clear from the relation
$$
\sum_{n=0}^{\infty}\frac{R_{n}'(x)}{n!}t^{n}=\frac{\partial F_{2}(x,t)}{\partial x}=\frac{t}{1-t}F_{2}(x,t).
$$
Thus, $R_{n}'(0)=1$ for $n\geq 1$. Moreover, from the recurrence relation satisfied by the sequence $(R_{n}(x))_{n\in\N}$, it is clear that $\op{deg}R_{n}=n$ and by writing $R_{n}(x)=\sum_{i=0}^{n}b(i,n)x^{i}$ we get $b(i,n)>0$ for $i\in\{1,\ldots,n\}$.
Finally, by comparing the coefficients near $x^{i}, i\in\{1,\ldots,n-1\}$, on both sides of the formula (\ref{formula}) we get the identity
\begin{equation}\label{idenL}
\frac{(i+1)a(i+1,n+2)}{n+1}=\sum_{j=0}^{\left\lfloor\frac{n-i}{2}\right\rfloor}\frac{(2j)!}{j!}\binom{n}{2j}h_{j}b(i,n-2j).
\end{equation}
In consequence $a(i,n)>0$ for each $i\in\{1,...,n-1\}$.
\end{proof}

\begin{rem}
{\rm Let us note that the sequence of polynomials $(R_{n}(x))_{n\in\N}$ coming from the power series expansion
$$
e^{\frac{xt}{1-t}}=\sum_{n=0}^{\infty}\frac{R_{n}(x)}{n!}t^{n},
$$
is also an interesting object. Indeed, if we write $R_{n}(x)=\sum_{i=0}^{n}b(i,n)x^{i}$, then the double array of coefficients $(b(i,n))$, where $n\in\N$ and $i\in\{0,\ldots,n\}$, can be easily found in OEIS database as the sequence A271703 \cite{Sloan}. The sequence is called the unsigned Lah numbers and has closed form $b(i,n)=\binom{n-1}{i-1}n!/i!$ for $n\in\N_{+}$ and $i\in\{1,\ldots,n-1\}$. It has many combinatorial interpretations. More precisely, the number $b(i,n)$ counts: (1) partially ordered sets on $n$ elements that consist entirely of $k$ chains; (2) number of ways to split the set $\{1,\ldots,n\}$ into an ordered collection of $n+1-i$ non-crossing nonempty sets; and (3) Dyck $n$-paths with $n+1-i$ peaks labeled $1,2,\ldots, n+1-i$ in some order, see \cite{Lus}. It would be nice to have combinatorial explanation of the identity (\ref{idenL}).
}
\end{rem}

\begin{rem}
{\rm The positivity of coefficients of the polynomial $V(n,x)$ is a bit unexpected property and the question arises whether similar phenomena hold for the case of $d\neq 2$. It seems that the direct generalization does not work. However, based on some numerical experiments we are able to formulate the following general
\begin{conj}
Let $d\in\N_{\geq 2}$ and define $U_{d}(n,x)=H_{d}(n-d,x)H_{d}(n+d,x)-H_{d}(n)^2$. Then, for $n\geq d$ there is a polynomial $V_{d}(n,x)\in \Z[x]$ of degree $\op{deg}V_{d}(n,x)=2\left\lfloor\frac{n}{d}\right\rfloor-1$ such that
$$
V_{d}(n,x^{d})=x^{-2(n\pmod{d})}U_{d}(n,x)
$$
and all coefficients of $V_{d}(n,x)$ are positive. In particular, for each $d\in\N_{\geq 2}$ and $n\geq d$ we have
$$
H_{d}(n-d)H_{d}(n+d)\geq H_{d}(n)^2.
$$
\end{conj}

It is possible to prove the above conjecture in the case $d=2$ by using the same approach as in the proof of Theorem \ref{positivity}. More precisely, one can prove the identity
$$
\cal{V}_{2}(x,t)=\frac{2 e^{\frac{tx}{1-t}} (2x+1-t)}{(1-t)(1-t^2) \sqrt{1-t^2}}=\sum_{n=0}^{\infty}\frac{U_{2}(n,\sqrt{x})}{n!}t^{n},
$$
and careful analysis of the convolution form of the polynomial $U_{2}(n,\sqrt{x})$ allows to prove that if $n\equiv 1\pmod{2}$, then all coefficients but $0$th are positive; in case of $n\equiv 0\pmod{2}$ all coefficients are positive. However, it seems that this approach can not be used in order to get the positivity of the coefficients of the polynomial $V_{d}(n,x)$ (or equivalently, the non-negativity of coefficients of the polynomial $U_{d}(n,x)$).

It is clear that our discussion is connected with the concept of concavity of sequences. More precisely, let us recall that the sequence ${\bf a}=(a_{n})_{n\in\N}$ is called {\it logarithmically concave} sequence, or a log-concave sequence for short, if $a_{n}^2\geq a_{n-1}a_{n+1}$ holds for all $n\geq k$ for some $k$. If the opposite inequality is satisfied for $n\geq k$, then the sequence is called {\it logarithmically convex} (log-convex for short). Let us observe that if the sequence ${\bf a}$ is positive, then the log-concavity of ${\bf a}$ implies the log-convexity of the sequence ${\bf a}^{-1}=(a_{n}^{-1})_{n\in\N}$ (and vice-versa of course). As an immediate consequence of Theorem \ref{positivity}, we get the following

\begin{prop}
If $x\geq 1$, then the sequence $(H_{2}(n,x))_{n\in\N}$ is log-convex. In particular, we have
$$
H_{2}(n)^2\leq H_{2}(n-1)H_{2}(n+1)
$$
and the sequence $(H_{2}(n))_{n\in\N}$ is log-convex.
\end{prop}

Let us note that from the work of Bender and Canfield \cite{BenCan} we know that for each $d\in\N_{\geq 2}$ and any given $x>0$, there is an integer $N_{d}(x)$ such that the sequence $(H_{d}(n,x))_{n\geq N_{d}(x)}$ is log-convex. However, the question concerning the non-negativity of coefficients of the polynomial $U_{d}(n,x)$ is of different nature. In fact, it is not even clear whether for given positive $x$, the inequality $U_{d}(n,x)\geq 0$ is true.}
\end{rem}

One can also note that the non-zero coefficients of the polynomial $U_{2}(n,x)$ are not too small. Indeed, this is a consequence of the following
\begin{prop}\label{Hd(n,x)mod(d-1)!}
Let $d\in\N_{\geq 2}$ be fixed. Then, for $n\geq d$, the following congruence is true
$$
H_{d}(n-d,x)H_{d}(n+d,x)\equiv H_{d}(n,x)^2\pmod{(d-1)!}.
$$
\end{prop}
\begin{proof}
The proof of the above congruence can be performed with the help of the basic recurrence relation satisfied by the sequence $(H_{d}(n,x))_{n\in\N}$ and the induction. We left the details for the reader.
\end{proof}

In fact, the following stronger result is true.

\begin{thm}
Let $d\in\N_{\geq 2}$. Then
$$
H_{d}(n-d,x)H_{d}(n+d,x)\equiv H_{d}(n,x)^2\pmod{d!},\quad n\geq d\quad \Leftrightarrow \quad d=p\quad\mbox{or}\quad d=p^2\quad\mbox{for}\quad p\in\mathbb{P}.
$$
\end{thm}

\begin{proof}
If $d$ is a prime number then by Theorem \ref{Hd(n,x)mod(d)} we have
$$H_d(n-d,x)H_d(n+d,x)\equiv x^{2(n\pmod{d})}(x-1)^{2\left\lfloor\frac{n}{d}\right\rfloor}\equiv H_d(n,x)^2\pmod{d},$$
which combined with the congruence $H_d(n-d,x)H_d(n+d,x)\equiv H_d(n,x)^2\pmod{(d-1)!}$ and the fact that $d$ and $(d-1)!$ are coprime, gives the congruence $H_d(n-d,x)H_d(n+d,x)\equiv H_d(n,x)^2\pmod{d!}$ for each integer $n\geq d$.

If $d$ is a composite number then at first we show the equivalence of the condition $$H_d(n-d,x)H_d(n+d,x)\equiv H_d(n,x)^2\pmod{d!}, \quad n\geq d$$ with the condition
\begin{align}\label{binomcong}
2{2d-1\choose d-1}\equiv 2\pmod{d}\mbox{ and }{n-d\choose d-1}+{n+d\choose d-1}\equiv 2{n\choose d-1}\pmod{d},\quad n\geq d.
\end{align}
For $n=d$ we have the following chain of equivalences:
\begin{equation}\label{initialcond}
\begin{aligned}
& H_d(0,x)H_d(2d,x)\equiv H_d(d,x)^2\pmod{d!}\\
\Leftrightarrow & 1\cdot\left(x^{2d}+\frac{(2d)_{(d)}}{d}x^d+\frac{(2d)_{(2d)}}{2d^2}\right)\equiv \left(x^d+\frac{(d)_{(d)}}{d}\right)^2\pmod{d!}\\
\Leftrightarrow & x^{2d}+2(2d-1)_{(d-1)}x^d+(2d-1)_{(d-1)}(d-1)!\equiv x^{2d}+2(d-1)!x^d+((d-1)!)^2\pmod{d!}\\
\Leftrightarrow & 2(2d-1)_{(d-1)}x^d+(2d-1)_{(d-1)}(d-1)!\equiv 2(d-1)!x^d+((d-1)!)^2\pmod{d!}\\
\Leftrightarrow & 2{2d-1\choose d-1}x^d+(2d-1)_{(d-1)}\equiv 2x^d+(d-1)!\pmod{d}.
\end{aligned}
\end{equation}
Since $(2d-1)_{(d-1)}\equiv (d-1)!\pmod{d}$ for any positive integer $d$, thus the last congruence is equivalent to $2{2d-1\choose d-1}\equiv 2\pmod{d}$. Now, let us assume that $n\geq d>4$ and $H_d(n-d,x)H_d(n+d,x)\equiv H_d(n,x)^2\pmod{d!}$. Then we have the following chain of equivalences:
\begin{equation}\label{inductioncond}
\begin{aligned}
& H_d(n+1-d,x)H_d(n+1+d,x)\equiv H_d(n+1,x)^2\pmod{d!}\\
\Leftrightarrow & (xH_d(n-d,x)+(n-d)_{(d-1)}H_d(n-2d+1,x))(xH_d(n+d,x)+(n+d)_{(d-1)}H_d(n+1,x))\\
& \equiv ((xH_d(n,x)+(n)_{(d-1)}H_d(n-d+1,x))^2\pmod{d!}\\
\Leftrightarrow & x^2H_d(n-d,x)H_d(n+d,x)+x(n-d)_{(d-1)}H_d(n-2d+1,x)H_d(n+d,x)\\
& +x(n+d)_{(d-1)}H_d(n-d,x)H_d(n+1,x)+(n-d)_{(d-1)}(n+d)_{(d-1)}H_d(n-2d+1,x)H_d(n+1,x)\\
& \equiv x^2H_d(n,x)^2+2x(n)_{(d-1)}H_d(n-d+1,x)H_d(n,x)+(n)_{(d-1)}^2H_d(n-d+1,x)^2\pmod{d!}\\
\Leftrightarrow & x(n-d)_{(d-1)}H_d(n-2d+1,x)H_d(n+d,x)+x(n+d)_{(d-1)}H_d(n-d,x)H_d(n+1,x)\\
& +(n-d)_{(d-1)}(n+d)_{(d-1)}H_d(n-2d+1,x)H_d(n+1,x)\\
& \equiv 2x(n)_{(d-1)}H_d(n-d+1,x)H_d(n,x)+(n)_{(d-1)}^2H_d(n-d+1,x)^2\pmod{d!}\\
\Leftrightarrow & x{n-d\choose d-1}H_d(n-2d+1,x)H_d(n+d,x)+x{n+d\choose d-1}H_d(n-d,x)H_d(n+1,x)\\
& +{n-d\choose d-1}(n+d)_{(d-1)}H_d(n-2d+1,x)H_d(n+1,x)\\
& \equiv 2x{n\choose d-1}H_d(n-d+1,x)H_d(n,x)+{n\choose d-1}(n)_{(d-1)}H_d(n-d+1,x)^2\pmod{d}\\
\Leftrightarrow & x{n-d\choose d-1}H_d(n-2d+1,x)H_d(n+d,x)+x{n+d\choose d-1}H_d(n-d,x)H_d(n+1,x)\\
& \equiv 2x{n\choose d-1}H_d(n-d+1,x)H_d(n,x)\pmod{d}\\
\Leftrightarrow & x{n-d\choose d-1}x^{n-2d+1}x^{n+d}+x{n+d\choose d-1}x^{n-d}x^{n+1}\equiv 2x{n\choose d-1} x^{n-d+1}x^{n}\pmod{d}\\
\Leftrightarrow & \left({n-d\choose d-1}+{n+d\choose d-1}\right)x^{2n-d+2}\equiv 2{n\choose d-1} x^{2n-d+2}\pmod{d}\\
\Leftrightarrow & {n-d\choose d-1}+{n+d\choose d-1}\equiv 2{n\choose d-1}\pmod{d},
\end{aligned}
\end{equation}
where we used Theorem \ref{Hd(n,x)mod(d)} and the fact that a composite number $d\geq 6$ divides a product of any consecutive $d-1$ integers.

Assume that $d=p^2$ for some odd prime number $p$. Then
\begin{align*}
& 2{2p^2-1\choose p^2-1}=2\prod_{j=1}^{p^2-1}\frac{p^2+j}{j}=2\prod_{j=1,p\nmid j}^{p^2-1}\frac{p^2+j}{j}\cdot\prod_{j=1}^{p-1}\frac{p^2+jp}{jp}=2\prod_{j=1,p\nmid j}^{p^2-1}\frac{p^2+j}{j}\cdot\prod_{j=1}^{p-1}\left(1+\frac{p}{j}\right)\\
& \equiv 2\prod_{j=1,p\nmid j}^{p^2-1}\frac{j}{j}\cdot\left(1+p\sum_{j=1}^{p-1}\frac{1}{j}\right)\equiv 2\left(1+p\sum_{j=1}^{p-1}j\right)=2\left(1+\frac{p^2(p-1)}{2}\right)\equiv 2\pmod{p^2}.
\end{align*}
In order to prove the congruence ${n-p^2\choose p^2-1}+{n+p^2\choose p^2-1}\equiv 2{n\choose p^2-1}\pmod{p^2}$ for $n\geq p^2$ we will compute ${n\choose p^2-1}$ modulo $p^2$. If $p^2\mid n+1$ then, similarly as above, we prove that ${n\choose p^2-1}\equiv 1\pmod{p^2}$. If $\nu_p(n+1)=1$ then we write $n+1=kp^2+lp$ for some $k,l\in\N_+$ with $l<p$ and compute
\begin{align*}
{n\choose p^2-1} & ={kp^2+lp-1\choose p^2-1}=\frac{kp^2}{lp}\prod_{j=1}^{lp-1}\frac{kp^2+j}{j}\cdot\prod_{j=lp+1}^{p^2-1}\frac{(k-1)p^2+j}{j}\\
& =\frac{kp}{l}\prod_{j=1, p\nmid j}^{lp-1}\frac{kp^2+j}{j}\cdot\prod_{j=1}^{l-1}\frac{kp^2+jp}{jp}\cdot\prod_{j=lp+1, p\nmid j}^{p^2-1}\frac{(k-1)p^2+j}{j}\cdot\prod_{j=l+1}^{p-1}\frac{(k-1)p^2+jp}{jp}\\
& =\frac{kp}{l}\prod_{j=1, p\nmid j}^{lp-1}\frac{kp^2+j}{j}\cdot\prod_{j=1}^{l-1}\frac{kp+j}{j}\cdot\prod_{j=lp+1, p\nmid j}^{p^2-1}\frac{(k-1)p^2+j}{j}\cdot\prod_{j=l+1}^{p-1}\frac{(k-1)p+j}{j}\\
& \equiv\frac{kp}{l}\prod_{j=1}^{l-1}\frac{kp+j}{j}\cdot\prod_{j=l+1}^{p-1}\frac{(k-1)p+j}{j}\equiv\frac{kp}{l}\prod_{j=1, j\neq l}^{p-1}\frac{j}{j}=\frac{kp}{l}\pmod{p^2}.
\end{align*}
If $p\nmid n+1$ then we write $n+1=kp^2+lp+m$ for some $k,m\in\N_+$ and $l\in\N$ with $l,m<p$ and compute
\begin{align*}
& {n\choose p^2-1}={kp^2+lp+m-1\choose p^2-1}=\frac{kp^2}{lp+m}\prod_{j=1}^{lp+m-1}\frac{kp^2+j}{j}\cdot\prod_{j=lp+m+1}^{p^2-1}\frac{(k-1)p^2+j}{j}\\
& =\frac{kp^2}{lp+m}\prod_{j=1, p\nmid j}^{lp+m-1}\frac{kp^2+j}{j}\cdot\prod_{j=1}^{l}\frac{kp^2+jp}{jp}\cdot\prod_{j=lp+m+1, p\nmid j}^{p^2-1}\frac{(k-1)p^2+j}{j}\cdot\prod_{j=l+1}^{p-1}\frac{(k-1)p^2+jp}{jp}\\
& =\frac{kp^2}{lp+m}\prod_{j=1, p\nmid j}^{lp+m-1}\frac{kp^2+j}{j}\cdot\prod_{j=1}^{l}\frac{kp+j}{j}\cdot\prod_{j=lp+m+1, p\nmid j}^{p^2-1}\frac{(k-1)p^2+j}{j}\cdot\prod_{j=l+1}^{p-1}\frac{(k-1)p+j}{j}\equiv 0\pmod{p^2}.
\end{align*}
Thus, independently on the $p$-adic valuation of $n+1$, there holds ${n-p^2\choose p^2-1}+{n+p^2\choose p^2-1}\equiv 2{n\choose p^2-1}\pmod{p^2}$. The only non-trivial case is $\nu_p(n+1)=1$. Then, writing $n+1=kp^2+lp$ for some $k,l\in\N_+$ with $l<p$, we obtain the following:
\begin{align*}
& {n-p^2\choose p^2-1}+{n+p^2\choose p^2-1}={(k-1)p^2+lp-1\choose p^2-1}+{(k+1)p^2+lp-1\choose p^2-1}\\
& \equiv\frac{(k-1)p}{l}+\frac{(k+1)p}{l}=\frac{2kp}{l}\equiv 2{kp^2+lp-1\choose p^2-1}\equiv 2{n\choose p^2-1}\pmod{p^2}.
\end{align*}
We have just proven the validity of congruences (\ref{binomcong}) for $d$ being a square of an odd prime number.

In the sequel, we assume that $d$ is a composite number not being a square of any prime number and satisfying congruences $H_{d}(n-d,x)H_{d}(n+d,x)\equiv H_{d}(n,x)^2\pmod{d!}$, $n\geq d$. Let us note that if congruences ${n-d\choose d-1}+{n+d\choose d-1}\equiv 2{n\choose d-1}\pmod{d}$ hold for each integer $n\geq d$ then, by the identity ${n+1\choose c+1}-{n\choose c+1}={n\choose c}$, the congruences ${n-d\choose c}+{n+d\choose c}\equiv 2{n\choose c}\pmod{d}$ are satisfied for each integers $n\geq d$ and $c\in\{0,...,d-1\}$. We have the following chain of congruences:
\begin{align*}
& {n-d\choose c}+{n+d\choose c}\equiv 2{n\choose c}\pmod{d}\\
\Leftrightarrow & \frac{1}{c!}\left((n-d)_{(c)}+(n+d)_{(c)}\right)\equiv\frac{2(n)_{(c)}}{c!}\pmod{d}\\
\Leftrightarrow & \frac{1}{c!}\left(\sum_{k=0}^c (-d)^k\sum_{0\leq j_1<...<j_{c-k}\leq c-1}(n-j_1)\cdot ...\cdot (n-j_{c-k})\right.\\
& \left. +\sum_{k=0}^c d^k\sum_{0\leq j_1<...<j_{c-k}\leq c-1}(n-j_1)\cdot ...\cdot (n-j_{c-k})\right)\equiv\frac{2(n)_{(c)}}{c!}\pmod{d}\\
\Leftrightarrow & \frac{2}{c!}\sum_{k=0}^{\left\lfloor\frac{c}{2}\right\rfloor}  d^{2k}\sum_{0\leq j_1<...<j_{c-2k}\leq c-1}(n-j_1)\cdot ...\cdot (n-j_{c-2k})\equiv\frac{2(n)_{(c)}}{c!}\pmod{d}
\end{align*}
\begin{align*}
\Leftrightarrow & \frac{2}{c!}\sum_{k=1}^{\left\lfloor\frac{c}{2}\right\rfloor}  d^{2k}\sum_{0\leq j_1<...<j_{c-2k}\leq c-1}(n-j_1)\cdot ...\cdot (n-j_{c-2k})\equiv 0\pmod{d}.
\end{align*}
Assume now that $d$ is a composite number not of the form $2p$, $2p^2$ for some odd prime number $p$ or $2^k$ for some integer $k\geq 2$. Then there is an odd prime number $p$ such that $d=tp^s$, where $s\in\N_+$ and $t\geq 3$ is an integer with $p$-adic valuation less than $s$. Taking $c=2p^s$ and $n=p^d-1$ in the expression $\frac{2}{c!}\sum_{k=1}^{\left\lfloor\frac{c}{2}\right\rfloor}  d^{2k}\sum_{0\leq j_1<...<j_{c-2k}\leq c-1}(n-j_1)\cdot ...\cdot (n-j_{c-2k})$, we obtain
\begin{align*}
& \frac{2}{(2p^s)!}\sum_{k=1}^{p^s} (tp^s)^{2k}\sum_{0\leq j_1<...<j_{c-2k}\leq c-1}(p^d-1-j_1)\cdot ...\cdot (p^d-1-j_{c-2k})\\
& =\frac{1}{(2p^s-1)_{(p^s-1)}(p^s-1)!}\sum_{k=1}^{p^s}  t^{2k}p^{2s(k-1)}\sum_{0\leq j_1<...<j_{c-2k}\leq c-1}(p^d-1-j_1)\cdot ...\cdot (p^d-1-j_{c-2k}).
\end{align*}
Since $d>s$, we have $\nu_p(p^d-1-i)=\nu_p(1+i)$ for any $i\in\{0,...,2p^s-1\}$. Hence, the summand with the least $p$-adic valuation is $$\frac{1}{(2p^s-1)_{(p^s-1)}(p^s-1)!}t^2(p^d-1)_{(p^s-1)}(p^d-p^s-1)_{(p^s-1)},$$ obtained for $k=1$, $j_i=i-1$, $i\in\{1,...,p^s\}$, and $j_i=i$, $i\in\{p^s+1,...,2p^s-2\}$. By the equality $\nu_p(p^d-1-i)=\nu_p(1+i)$ for any $i\in\{0,...,2p^s-1\}$ this summand has $p$-adic valuation equal to $2\nu_p(t)<s+\nu_p(t)=\nu_p(d)$. Any other summand has strictly greater $p$-adic valuation because some from the factors $p^d-1,...,p^d-p^s+1,p^d-p^s-1,...,p^d-2p^s+1$, having $p$-adic valuations less than $s$, are replaced by $p^d-p^s$, $p^d-2p^s$ or $d=tp^s$ with $p$-adic valuations at least equal to $s$. Thus, the whole sum $\frac{1}{(2p^s-1)_{(p^s-1)}(p^s-1)!}\sum_{k=1}^{p^s}  t^{2k}p^{2s(k-1)}\sum_{0\leq j_1<...<j_{c-2k}\leq c-1}(p^d-1-j_1)\cdot ...\cdot (p^d-1-j_{c-2k})$ has $p$-adic valuation less than $\nu_p(d)$, which implies that it cannot be divisible by $d$ - a contradiction.

We are checking now the numbers $d$ of the form $2p$, $2p^2$ for some odd prime number $p$. We compute the value ${2d-1\choose d-1}\pmod{p}$. For shortening the proof we will perform the computation for numbers $d=2p^2$ as for the remaining case the computations are simpler.
\begin{align*}
& {4p^2-1\choose 2p^2-1}=\prod_{j=1}^{2p^2-1}\frac{2p^2+j}{j}=\prod_{j=1, p\nmid j}^{2p^2-1}\frac{2p^2+j}{j}\cdot\prod_{j=1, p\nmid j}^{2p-1}\frac{2p^2+jp}{jp}\cdot\frac{3p^2}{p^2}\\
& =\prod_{j=1, p\nmid j}^{2p^2-1}\frac{2p^2+j}{j}\cdot\prod_{j=1, p\nmid j}^{2p-1}\frac{2p+j}{j}\cdot 3\equiv \prod_{j=1, p\nmid j}^{2p^2-1}\frac{j}{j}\cdot\prod_{j=1, p\nmid j}^{2p-1}\frac{j}{j}\cdot 3\equiv 3\pmod{p}
\end{align*}
Then $2{4p^2-1\choose 2p^2-1}\equiv 6\pmod{p}$ but, on the other hand, there should be $2{4p^2-1\choose 2p^2-1}\equiv 2\pmod{p}$. This implies the congruence $6\equiv 2\pmod{p}$ and, as a result, $p\mid 4$ - a contradiction, as $p$ is an odd prime number.

Let $d=2^k$ for an integer $k\geq 3$. We compute the value ${2d-1\choose d-1}\pmod{p}$:
\begin{align*}
& {2^{k+1}-1\choose 2^k-1}=\prod_{j=1}^{2^k-1}\frac{2^k+j}{j}=\frac{3\cdot 2^{k-1}}{2^{k-1}}\cdot\frac{5\cdot 2^{k-2}}{2^{k-2}}\cdot\frac{7\cdot 2^{k-2}}{3\cdot 2^{k-2}}\prod_{l=0}^{k-3}\prod_{j=1, 2\nmid j}^{2^{k-l}-1}\frac{2^k+2^lj}{2^lj}\\
& =3\cdot 5\cdot\frac{7}{3}\prod_{l=0}^{k-3}\prod_{j=1, 2\nmid j}^{2^{k-l}-1}\frac{2^{k-l}+j}{j}\equiv 35\prod_{l=0}^{k-3}\prod_{j=1, 2\nmid j}^{2^{k-l}-1}\frac{j}{j}\equiv 3\pmod{8}.
\end{align*}
Then $2{2^{k+1}-1\choose 2^k-1}\equiv 6\pmod{8}$ but, on the other hand, there should be $2{2^{k+1}-1\choose 2^k-1}\equiv 2\pmod{8}$. This implies the congruence $6\equiv 2\pmod{8}$ and, as a result, $8\mid 4$ - a contradiction.

We are left with the case $d=4$. By (\ref{initialcond}), the condition $H_4(0,x)H_4(8,x)\equiv H_4(4,x)^2\pmod{4!}$ is equivalent to the condition $2{7\choose 3}\equiv 2\pmod{4}$, which holds. It is easy to compute that
\begin{equation}\label{(n/3)mod4}
{n\choose 3}\equiv
\begin{cases}
2\left\lfloor\frac{n}{4}\right\rfloor +1, & \mbox{ if } 4\mid n+1\\
2, & \mbox{ if } \nu_2(n+1)=1\\
0, & \mbox{ if } \nu_2(n+1)=0
\end{cases}
\pmod{4}.
\end{equation}
Assuming that $H_4(n-4,x)H_4(n+4,x)\equiv H_4(n,x)^2\pmod{4!}$, we prove similarly as in (\ref{inductioncond}) that the congruence $H_4(n-3,x)H_4(n+5,x)\equiv H_4(n+1,x)^2\pmod{4!}$ is equivalent to
\begin{align*}
& x{n-4\choose 3}H_4(n-7,x)H_4(n+4,x)+x{n+4\choose 3}H_4(n-4,x)H_4(n+1,x)\\
& +{n-4\choose 3}(n+4)_{(3)}H_4(n-7,x)H_4(n+1,x)\\
& \equiv 2x{n\choose 3}H_4(n-3,x)H_4(n,x)+{n\choose 3}(n)_{(3)}H_4(n-3,x)^2\pmod{4}.
\end{align*}
Further, the above condition is equivalent to the following equivalent congruences:
\begin{align*}
& x{n-4\choose 3}\left(x^{n-7}+2\left\lfloor\frac{n-7}{4}\right\rfloor x^{n-11}\right)\left(x^{n+4}+2\left\lfloor\frac{n+4}{4}\right\rfloor x^{n}\right)\\
& +x{n+4\choose 3}\left(x^{n-4}+2\left\lfloor\frac{n-4}{4}\right\rfloor x^{n-8}\right)\left(x^{n+1}+2\left\lfloor\frac{n+1}{4}\right\rfloor x^{n-3}\right)\\
& +{n-4\choose 3}(n+4)_{(3)}\left(x^{n-7}+2\left\lfloor\frac{n-7}{4}\right\rfloor x^{n-11}\right)\left(x^{n+1}+2\left\lfloor\frac{n+1}{4}\right\rfloor x^{n-3}\right)\\
& \equiv 2x{n\choose 3}\left(x^{n-3}+2\left\lfloor\frac{n-3}{4}\right\rfloor x^{n-7}\right)\left(x^{n}+2\left\lfloor\frac{n}{4}\right\rfloor x^{n-4}\right)\\
& +{n\choose 3}(n)_{(3)}\left(x^{n-3}+2\left\lfloor\frac{n-3}{4}\right\rfloor x^{n-7}\right)^2\pmod{4}\\
\Leftrightarrow & {n-4\choose 3}\left[x^{2n-2}+2\left(\left\lfloor\frac{n-7}{4}\right\rfloor +\left\lfloor\frac{n+4}{4}\right\rfloor\right) x^{2n-6}\right]\\
& +{n+4\choose 3}\left[x^{2n-2}+2\left(\left\lfloor\frac{n-4}{4}\right\rfloor +\left\lfloor\frac{n+1}{4}\right\rfloor\right) x^{2n-6}\right]\\
& +{n-4\choose 3}(n+4)_{(3)}\left[x^{2n-6}+2\left(\left\lfloor\frac{n-7}{4}\right\rfloor +\left\lfloor\frac{n+1}{4}\right\rfloor\right) x^{2n-10}\right]\\
& \equiv 2{n\choose 3}\left[x^{2n-2}+2\left(\left\lfloor\frac{n-3}{4}\right\rfloor +\left\lfloor\frac{n}{4}\right\rfloor\right) x^{2n-6}\right]+{n\choose 3}(n)_{(3)}x^{2n-6}\pmod{4}.
\end{align*}
The last congruence is true, because $2\left(\left\lfloor\frac{n-7}{4}\right\rfloor +\left\lfloor\frac{n+4}{4}\right\rfloor\right)\equiv 2\left(\left\lfloor\frac{n-4}{4}\right\rfloor +\left\lfloor\frac{n+1}{4}\right\rfloor\right)\equiv 2\left(\left\lfloor\frac{n-3}{4}\right\rfloor +\left\lfloor\frac{n}{4}\right\rfloor\right)\pmod{4}$, $\left(\left\lfloor\frac{n-7}{4}\right\rfloor +\left\lfloor\frac{n+1}{4}\right\rfloor\right)\equiv 0\pmod{2}$ and by (\ref{(n/3)mod4}) we have ${n-4\choose 3}+{n+4\choose 3}\equiv 2{n\choose 3}\pmod{4}$ and ${n-4\choose 3}(n+4)_{(3)}\equiv {n\choose 3}(n)_{(3)}\pmod{4}$. Hence, $H_4(n-4,x)H_4(n+4,x)\equiv H_4(n,x)^2\pmod{4!}$ for each integer $n\geq 4$.
\end{proof}

\bigskip

Two further identities involving elements of the sequence $(H_{d}(n,x))_{n\in\N}$ can be proved. More precisely, we have the following:

\begin{thm}
Let $d\in\N_{\geq 2}$ be given. For each $n\in\N$ the following identities hold
\begin{align*}
 H_{d}(n+d,x)H_{d}(n+d-2,x)-H_{d}(n+d-1,x)^2&=(n+d-3)_{(d-3)}(H_{d}(n,x)H_{d}'(n+d-1,x)\\
 &-H_{d}(n+d-1,x)H_{d}'(n,x)),\\
 H_{d}(n,x)H_{d}''(n+2,x)-H_{d}'(n+2,x)H_{d}''(n,x)&=(n+2)(H_{d}(n,x)H_{d}''(n+1,x)\\
 &-H_{d}(n+1,x)H_{d}''(n,x).
\end{align*}
\end{thm}

\section{Some related observations, questions and conjectures}\label{Section9}

In this section we collect observations concerning some related sequences connected with the family of sequences $(H_{d}(n))_{n\in\N}$ and formulate several questions and conjectures appeared during our investigations.

\bigskip

In \cite{AmdMoll} the authors introduced the summatory sequence for the sequence of involutions $(H_{2}(n))_{n\in\N}$. Thus, it is natural to consider the sequence $(G_{d}(n))_{n\in\N}$ for $d\in\N_{\geq 2}$ with
$$
G_{d}(n)=\sum_{i=0}^{n}H_{d}(i).
$$
From the definition, the number $G_{d}(n)$ can be seen as a number of all permutations in $\coprod_{i=1}^{n}S_{i}$ which are products of cycles of length $d$.

Let us recall the following useful result.

\begin{lem}[Lemma 5.2 in \cite{AmdMoll}]\label{AMlem}
If $A(x)=\sum_{n=0}^{\infty}\frac{a_{n}}{n!}x^{n}$ and $s_{n}=\sum_{i=0}^{n}a_{i}$ then
$$
\sum_{n=0}^{\infty}\frac{s_{n}}{n!}x^{n}=A(x)+e^{x}\int_{0}^{x}e^{-t}A(t)dt.
$$
\end{lem}

We are ready to prove the following

\begin{thm}\label{GenGd}
\begin{enumerate}
\item We have
$$
\cal{G}_{d}(x)=\sum_{n=0}^{\infty}\frac{G_{d}(n)}{n!}x^{n}=\cal{H}_{d}(x)+e^{x}\int_{0}^{x}e^{\frac{t^{d}}{d}}dt.
$$
\item The sequence $(G_{d}(n))_{n\in\N}$ satisfies the following recurrence relation: $G_{d}(i)=i+1$ for $i=-1,\ldots, d-1$, and for $n\geq d$ we have
\begin{equation}\label{recforGd}
G_{d}(n)=2G_{d}(n-1)+G_{d}(n-2)+(n-1)_{(d-1)}(G_{d}(n-d)-G_{d}(n-d-1)).
\end{equation}
\end{enumerate}
\end{thm}
\begin{proof}
In order to get the exponential generating function for the sequence $(G_{d}(n))_{n\in\N}$ we apply Lemma \ref{AMlem} with $A(x)=\cal{H}_{d}(x)$ and get the result.

The recurrence relation for the sequence $(G_{d}(n))_{n\in\N}$ can be easily proved by noting the identity $H_{d}(n)=G_{d}(n)-G_{d}(n-1)$. Indeed, from the relation (\ref{basicrec}) we get
$$
G_{d}(n)-G_{d}(n-1)=G_{d}(n-1)-G_{d}(n-2)+(n-1)_{(d-1)}(G_{d}(n-d)-G_{d}(n-d-1))
$$
and hence the result.
\end{proof}
We can apply the previous result and get the following general fact (with special case for $d=2$ obtained by Amdeberhan and Moll).

\begin{cor}
We have the following identity:
$$
\sum_{i=0}^{n}(-1)^{n-i}\binom{n}{i}G_{d}(i-1)=
\begin{cases}
\begin{array}{lll}
0,                    & &\mbox{ if } n\not\equiv 1\pmod{d}\\
\frac{(md)!}{d^{m}m!},& &\mbox{ if } n=dm+1
\end{array}
\end{cases}.
$$
\end{cor}
\begin{proof}
We have $\cal{H}_{d}(x)=e^{x+\frac{x^{d}}{d}}$ and using the first part of Theorem \ref{GenGd} we get
$$
\int_{0}^{x}e^{\frac{t^{d}}{d}}dt=e^{-x}(\cal{G}_{d}(x)-\cal{H}_{d}(x))=e^{-x}\sum_{n=1}^{\infty}\frac{G_{d}(n-1)}{n!}x^{n}=\sum_{n=0}^{\infty}\left(\sum_{i=0}^{n}(-1)^{n-i}\binom{n}{i}G_{d}(i-1)\right)\frac{x^{n}}{n!}.
$$
Comparing the coefficients on both sides of the first and the last expression we get the result.
\end{proof}

In the paper \cite{AmdMoll} the authors were able to compute exact expression for the 2-adic valuation of $G_{2}(n)$. One can ask whether it is possible to compute $\nu_{p}(G_{p}(n))$ for $p\in\mathbb{P}$ and $n\in\N_{+}$. A numerical computations suggest the following:

\begin{conj}
\begin{enumerate}
\item[(1)] We have the following expression
 $$
 \nu_{3}(G_{3}(9n+k))=\begin{cases}
\begin{array}{lll}
2n &  & \mbox{for}\;k=0,1,\\
2n+1 &  & \mbox{for}\;k=2,3,4,\\
2n+2 &  & \mbox{for}\;k=5,6,7,\\
2n+5+\nu_{3}(n+1) &  & \mbox{for}\;k=8.
\end{array}
\end{cases}
 $$

\item[(2)] We have the following expression
 $$
 \nu_{5}(G_{5}(25n+k))=\begin{cases}
\begin{array}{lll}
4n &  & \mbox{for}\;k=0,1,2,3,\\
4n+1 &  & \mbox{for}\;k=4,5,7,\\
4n+2 &  & \mbox{for}\;k=6,8,9,10,11,\\
4n+3 &  & \mbox{for}\;k=12,13,14,16,17,18,\\
4n+4 &  & \mbox{for}\;k=15,20,21,22,23,\\
4n+5 &  &\mbox{for}\;k=19,\\
4n+6+\nu_{5}(n+1) & & \mbox{for}\;k=24
\end{array}
\end{cases}
 $$
\end{enumerate}
\end{conj}

It is very likely that the  method employed by Amdeberhan and Moll can be extended in order to confirm the above conjecture. However, it would be nice to have general method which works for all $p\in\mathbb{P}$.

\bigskip

Let $d\in\N_{\geq 2}$ be fixed and for $n\in\N_{+}$ define the matrix
$$
M_{d}(n,x)=[H_{d}(i+j,x)]_{0\leq i, j\leq n-1}.
$$
Based on numerical calculations we formulate the following

\begin{conj}
Let $d\in\N_{\geq 2}$.
\begin{enumerate}
\item For $n\in\N_{+}$ we have $\op{det}M_d(n,x)\in\Z$ and
$$\op{det}M_d(n,x)\equiv 0\pmod{\prod_{i=0}^{n-1}i!};$$
\item For $n\in\N_{+}$ we have the identity
$$
\op{det}M_{2}(n,x)=\prod_{i=0}^{n-1}i!\;.
$$
\item If $d\geq 3$, then we have the following property
$$
\op{det}M_{d}(n,x)=0 \Leftrightarrow n\equiv 2,\ldots, d-1\pmod{d}.
$$
\end{enumerate}
\end{conj}
\bigskip

Let $a,b\in\N_{\geq 2}$ with $a<b$. It seems to be interesting to find out if there are some integers $n\geq a$ and $m\geq b$ such that
\begin{equation}\label{HanHbm}
H_a(n)=H_b(m).
\end{equation}
We only know that if $a$ is a prime number then the equation (\ref{HanHbm}) has no solutions, as $a$ divides the left hand side and does not divide the right hand side.

\begin{ques}
Is there any quadruple $(a,b,n,m)\in\N^4$ such that $2\leq a<b$, $n\geq a$, $m\geq b$ and $H_a(n)=H_b(m)$? If yes, are there infinitely many such quadruples?
\end{ques}

\bigskip

\noindent Jagiellonian University, Faculty of Mathematics and Computer Science, Institute of Mathematics, {\L}ojasiewicza 6, 30 - 348 Krak\'{o}w, Poland;

\noindent email: {\tt \{piotr.miska, maciej.ulas\}@uj.edu.pl}

\end{document}